\documentclass[fleqn,reqno,11pt,a4paper,final]{amsart}

\usepackage[a4paper,left=20mm,right=20mm,top=20mm,bottom=20mm,marginpar=20mm]{geometry}
\usepackage{amsmath}
\usepackage{amssymb}
\usepackage{amsthm}
\usepackage{amscd}
\usepackage[utf8]{inputenc}
\usepackage{bbm}
\usepackage{color}
\usepackage[english=american]{csquotes}
\usepackage[final]{graphicx}
\usepackage{hyperref}
\usepackage{calc}
\usepackage{mathptmx}
\usepackage{mathtools}
\usepackage{tikz}
\usepackage{graphicx}
\usepackage{enumitem}
\usepackage{stix}

\usepackage[backend=biber,bibencoding=utf8,firstinits,maxnames=4,style=alphabetic,isbn=false, doi=false, eprint= true, url= false,date=year]{biblatex}
\bibliography{literature.bib}

\definecolor{luh-dark-blue}{rgb}{0.0, 0.313, 0.608}
\definecolor{luh-light-blue}{rgb}{0.6, 0.725, 0.847}
\definecolor{lightblue}{RGB}{160,220,255}

\graphicspath{{../Pictures/}}

\numberwithin{equation}{section}

\newtheoremstyle{thmlemcorr}{10pt}{10pt}{\itshape}{}{\bfseries}{.}{10pt}{{\thmname{#1}\thmnumber{ #2}\thmnote{ (#3)}}}
\newtheoremstyle{thmlemcorr*}{10pt}{10pt}{\itshape}{}{\bfseries}{.}\newline{{\thmname{#1}\thmnumber{ #2}\thmnote{ (#3)}}}
\newtheoremstyle{remexample}{10pt}{10pt}{}{}{\bfseries}{.}{10pt}{{\thmname{#1}\thmnumber{ #2}\thmnote{ (#3)}}}
\newtheoremstyle{ass}{10pt}{10pt}{}{}{\bfseries}{.}{10pt}{{\thmname{#1}\thmnumber{ A#2}\thmnote{ (#3)}}}

\theoremstyle{thmlemcorr}
\newtheorem{theorem}{Theorem}
\numberwithin{theorem}{section}
\newtheorem{lemma}[theorem]{Lemma}
\newtheorem{corollary}[theorem]{Corollary}
\newtheorem{proposition}[theorem]{Proposition}

\theoremstyle{thmlemcorr*}
\newtheorem*{theorem*}{Theorem}
\newtheorem{lemma*}[theorem]{Lemma}
\newtheorem{corollary*}[theorem]{Corollary}
\newtheorem{proposition*}[theorem]{Proposition}
\newtheorem{problem*}[theorem]{Problem}
\newtheorem{conjecture*}[theorem]{Conjecture}
\newtheorem{definition*}[theorem]{Definition}
\newtheorem{assumption*}[theorem]{Assumption}

\theoremstyle{remexample}
\newtheorem{remark}[theorem]{Remark}

\theoremstyle{ass}

\DeclareMathOperator{\diverg}{div}

\newcommand{\dd}{\;\mathrm{d}}

\newcommand{\N}{\mathbb{N}}
\newcommand{\R}{\mathbb{R}}

\newcommand{\eps}{\varepsilon}

\def\XXint#1#2#3{{\setbox0=\hbox{$#1{#2#3}{\int}$}
\vcenter{\hbox{$#2#3$}}\kern-.5\wd0}}

\renewcommand{\eps}{\varepsilon}
\renewcommand{\phi}{\varphi}

\begin{document}


\title[]{Bernis estimates for higher-dimensional doubly-degenerate non-Newtonian thin-film equations}

\author{Christina Lienstromberg}
\address{\textit{Christina Lienstromberg:}  Institute of Analysis, Dynamics and Modeling, University of Stuttgart, Pfaffenwaldring~57, 70569 Stuttgart, Germany}
\email{christina.lienstromberg@mathematik.uni-stuttgart.de}

\author{Katerina Nik}
\address{\textit{Katerina Nik:} King Abdullah University of Science and Technology (KAUST), CEMSE Division, Thuwal 23955-6900, Saudi Arabia}
\email{katerina.nik@kaust.edu.sa}


\begin{abstract}
For the doubly-degenerate parabolic non-Newtonian thin-film equation
\begin{equation*}
    u_t + \diverg\bigl(u^n |\nabla \Delta u|^{p-2} \nabla \Delta u\bigr) = 0,
\end{equation*}
we derive (local versions) of Bernis estimates of the form
 \begin{equation*}
        \int_{\Omega}  u^{n-2p} |\nabla u|^{3p} \dd x  
        + 
        \int_{\Omega}  u^{n-\frac{p}{2}} |\Delta u|^{\frac{3p}{2}} \dd x 
        \leq c(n,p,d) \int_{\Omega}  u^n|\nabla \Delta u|^p \dd x,
\end{equation*}
for functions $u \in W^2_p(\Omega)$ with Neumann boundary condition, where $2 \leq p < \frac{19}{3}$ and $n$ lies in a certain range. 
Here, $\Omega \subset \R^d$ is a smooth convex domain with $d < 3p$. 
A particularly important consequence is the estimate
\begin{equation*}
    \int_{\Omega} 
    |\nabla \Delta (u^{\frac{n+p}{p}})|^p\dd x \leq c(n,p,d) \int_{\Omega} u^n|\nabla \Delta u|^p \dd x.
\end{equation*}
The methods used in this article follow the approach of \cite{gruen2001} for the Newtonian case, while addressing the specific challenges posed by the nonlinear higher-order term $|\nabla \Delta u|^{p-2} \nabla \Delta u$ and the additional degeneracy. The derived estimates are key to establishing further qualitative results, such as the existence of weak solutions, finite propagation of support, and the appearance of a waiting-time phenomenon. 
\end{abstract}
\vspace{4pt}

\maketitle






\section{Introduction and main results}

\subsection{Background} 
In this article, we derive so-called Bernis estimates for fourth-order doubly-degenerate parabolic problems of the form
\begin{equation} \label{eq:PDE}
    \begin{cases}
        u_t + 
        \diverg
        \bigl(
        u^{n} 
        |\nabla \Delta u|^{p-2} \nabla \Delta u
        \bigr)
        = 0
        & \quad \text{on } (0,\infty)\times \Omega \subset \R^{d+1},
        \\
        \nabla u\cdot \nu = \nabla \Delta u\cdot \nu 
        = 0
        &
        \quad \text{in } (0,\infty)\times \partial \Omega,
         \\
         u(0,\cdot)
         =
         u_0(\cdot)
         &
        \quad \text{in } \Omega,
    \end{cases}
\end{equation}
as they arise in the modeling of non-Newtonian thin-film flows with a strain rate-dependent power-law rheology. 

Problem \eqref{eq:PDE} describes the evolution of the height $u = u(t,x)\geq 0$ of a thin layer of a viscous, non-Newtonian, and incompressible fluid over an impermeable flat bottom, as sketched in Figure \ref{fig:thin-film}.

The partial differential equation $\eqref{eq:PDE}_1$ is a fourth-order quasilinear equation that is doubly-degenerate parabolic, meaning that parabolicity gets lost when either the film height $u$ or the third-order term $\nabla \Delta u$ vanishes. 

The term $u^n$ represents the mobility of the flow, with the exponent $n > 0$ determined by the choice of the slip condition at the lower boundary (i.e., the interface between the solid bottom and the fluid film). For a no-slip condition, we have $n=p+1$, while for a Navier-slip condition, we have $n=p$. Other values of $n \in (p-2,p+1)$ have been proposed to model weaker or stronger slippage; see, for example, \cite{greenspan1978}. The term $|\nabla \Delta u|^{p-2} \nabla \Delta u$ has a dual physical interpretation: the higher-order term $\nabla \Delta u$ accounts for surface tension effects (we assume that the dynamics of the flow is driven solely by surface tension, thus neglecting gravitational effects), while its appearance in the nonlinear form $|s|^{p-2}s$ is due to the non-Newtonian power-law rheology of the fluid. 

The evolution equation $\eqref{eq:PDE}_1$ is supplemented with  Neumann-type boundary conditions $\nabla u \cdot \nu = 0$ and $\nabla\Delta u \cdot \nu = 0$, which represent a zero-contact angle condition and a no-flux condition at the lateral boundary $(0,\infty)\times\partial\Omega$, respectively. 
Finally, $u_0$ denotes the initial film height. 


\begin{center}
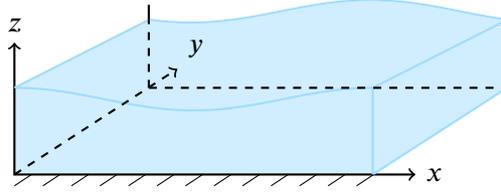
\begin{figure}[h]
\begin{tikzpicture}[domain=0:2*pi, xscale=0.75, yscale=0.5] 
\draw[ultra thick, smooth, variable=\x, lightblue] plot[domain=0.75*pi:2.75*pi] (\x,{0.3*cos(\x r + 0.75*pi)+4.3}); 
\fill[lightblue!50] plot[domain=0:2*pi] (\x,{0}) -- plot[domain=2*pi:0] (\x,{0.3*cos(\x r)+2}); 
\fill[lightblue!50] plot[domain=2*pi:0] (\x,{0.3*cos(\x r)+2}) -- plot[domain=0.75*pi:2.75*pi] (\x,{0.3*cos(\x r + 0.75*pi)+4.3}); 
\fill[lightblue!50] plot[domain=2*pi:0] (\x,{0.3*cos(\x r)+2}) -- plot[domain=0.75*pi:2.75*pi] (\x,{0.3*cos(\x r + 0.75*pi)+4.3}); 
\draw[thick, smooth, variable=\x, lightblue] plot (\x,{0.3*cos(\x r)+2});  
\draw[thick,<->] (2*pi+0.75,0) node[right] {$x$} -- (0,0) -- (0,3.5) node[above] {$z$};
\fill[lightblue!50] (2*pi,0) -- (2*pi,2.3) -- (2.75*pi,4.15) -- (2.75*pi,2.3) -- (2*pi,0); 
\draw[thick,lightblue] (2*pi,0) -- (2*pi,2.3) -- (2.75*pi,4.1) -- (2.75*pi,2.3) -- (2*pi,0); 
\draw[thick,lightblue] (0,2.3) -- (0.75*pi,4.1);
\draw[thick,dashed,->] (0,0) -- (0.75*pi+0.5,2.3+0.5) node[above right] {$y$};
\draw[thick,dashed] (0.75*pi,2.3) -- (0.75*pi,4.1);
\draw[thick] (0.75*pi,4.1) -- (0.75*pi,4.5);
\draw[thick,dashed] (0.75*pi,2.3) -- (2.75*pi,2.3);
\draw[-] (0,-0.3) -- (0.3, 0);
\draw[-] (0.5,-0.3) -- +(0.3, 0.3);
\draw[-] (1,-0.3) -- +(0.3, 0.3);
\draw[-] (1.5,-0.3) -- +(0.3, 0.3);
\draw[-] (2,-0.3) -- +(0.3, 0.3); 
\draw[-] (2.5,-0.3) -- +(0.3, 0.3);
\draw[-] (3,-0.3) -- +(0.3, 0.3);
\draw[-] (3.5,-0.3) -- +(0.3, 0.3);
\draw[-] (4,-0.3) -- +(0.3, 0.3);
\draw[-] (4.5,-0.3) -- +(0.3, 0.3);
\draw[-] (5,-0.3) -- +(0.3, 0.3);
\draw[-] (5.5,-0.3) -- +(0.3, 0.3);
\draw[-] (6,-0.3) -- +(0.3, 0.3);
\end{tikzpicture}   
\caption{Sketch of a fluid film on an impermeable solid bottom.}
\label{fig:thin-film}
\end{figure} 
\end{center}

The natural energy inequality for non-negative solutions to \eqref{eq:PDE} is given by 
\begin{equation}\label{eq:energy_inequ}
    \frac{1}{2}\int_\Omega |\nabla u(t)|^2 \dd x
    +
    \int_0^t\int_\Omega u^n |\nabla \Delta u|^p \dd x \dd s
    \leq
    \frac{1}{2}\int_\Omega |\nabla u_0|^2 \dd x, \quad t \geq 0.
\end{equation}
However, even in the one-dimensional case or for the classical Newtonian thin-film equation ($p=2$), this estimate on the dissipation $\int_\Omega u^n |\nabla \Delta u|^p \dd x$ is not sufficient for many purposes. Further estimates on the integrals $\int_\Omega |\nabla\big(u^\frac{n+p}{3p}\big)|^{3p} \dd x$, $\int_{\Omega} |\Delta\big(u^{\frac{2(n+p)}{3p}}\big)|^{\frac{3p}{2}} \dd x$, and $\int_\Omega |\nabla\Delta\big(u^\frac{n+p}{p}\big)|^{p} \dd x$ are needed to obtain qualitative results such as finite propagation of support or the occurrence of a waiting-time phenomenon. See Section \ref{sec:related_results} for more details on related results.

\medskip


\subsection{Aim of the article and main results}
The main goal of this article is to derive inequalities of the form
\begin{equation*}
    \begin{split}
        \int_\Omega \varphi^{3p} u^{n-2p} |\nabla u|^{3p} \dd x
        +
        \int_\Omega \varphi^{3p} u^{n-\frac{p}{2}} |\Delta u|^\frac{3p}{2} \dd x 
        \leq
        c(n,p,d) \left(
        \int_\Omega \varphi^{3p} u^n |\nabla \Delta u|^p \dd x
        +
        \int_\Omega |\nabla \varphi|^{3p} u^{n+p} \dd x
        \right)
    \end{split}
\end{equation*}
for a given non-negative test function $\phi$, and for positive functions $u \in W^2_p(\Omega),\ 2 \leq p < \frac{19}{3}$, satisfying the Neumann boundary condition $\nabla u \cdot \nu=0$ on $\partial \Omega$ and having finite dissipation $\int_\Omega  u^n |\nabla \Delta u|^p \dd x <\infty$. These are established in Theorem \ref{thm:Bernis} below. 

From these inequalities,  we can deduce further estimates of the form 
\begin{equation*}
    \begin{split}
        &\int_{\Omega} \Big|\nabla \big(u^{\frac{n+p}{3p}}\big)\Big|^{3p} \dd x + 
        \int_{\Omega} \Big|\Delta\big(u^{\frac{2(n+p)}{3p}}\big)\Big|^{\frac{3p}{2}} \dd x
        +
        \int_{\Omega}   \Big|\nabla \Delta\big(u^{\frac{n+p}{p}}\big)\Big|^{p} \dd x  
        \leq c(n,p,d) \int_{\Omega} u^n|\nabla \Delta u|^p \dd x
    \end{split}
\end{equation*}
for $u$ as above. This result is proven in Theorem \ref{thm:consequence} below. 

These inequalities, established in Theorems \ref{thm:Bernis} and \ref{thm:consequence}, provide a basis for applying appropriate interpolation results, which are in turn crucial for proving qualitative properties of solutions to \eqref{eq:PDE}, such as  finite speed of propagation or waiting-time phenomena.

Our results generalize the findings of \cite{gruen2001}, where the Newtonian analogue of \eqref{eq:PDE} is addressed, to non-Newtonian thin-film equations of the form \eqref{eq:PDE}. In the one-dimensional non-Newtonian case, similar Bernis inequalities were shown in \cite{ansini2004} in the parameter range $p \geq \frac{4}{3}$ and $\frac{p-1}{2} < n < 2p-1$.

In the following, we present the main results of this article, Theorems \ref{thm:Bernis} and \ref{thm:consequence}, in detail.

\begin{theorem}\label{thm:Bernis}
Let $\Omega\subset \R^d$, $d < 3p$,  be a bounded, convex domain with smooth boundary $\partial\Omega$. Let $ 2\leq p < \frac{19}{3}$ and 
\begin{equation*}
    2p - 2\frac{2 + \sqrt{(3p-2)^2 + d(3p-4)}}{-3p+6 + \sqrt{(3p-2)^2 + d(3p-4)}} 
    < n 
    < 2p-1.
\end{equation*}
Assume $u\in W^2_p(\Omega)$ is positive, satisfies $\nabla u \cdot \nu = 0$ on $\partial\Omega$, and
\begin{equation*}
    \int_\Omega u^n |\nabla \Delta u|^p\, dx < \infty.
\end{equation*}
Then there exist positive constants $c_i=c_i(n,p,d)>0$,  $i=1,2,3$, 
such that the estimates
\begin{equation} \label{eq:Bernis_new_I}
    \begin{split}
        \int_{\Omega} \phi^{3p}u^{n-2p} |\nabla u|^{3p} \dd x \leq c_{1}\left( \int_{\Omega} \phi^{3p} u^n |\nabla \Delta u|^p \dd x + \int_{\Omega} |\nabla \phi|^{3p} u^{n+p} \dd x \right)
    \end{split}
\end{equation}
and
\begin{equation} \label{eq:Bernis_new_II}
    \begin{split}
        &
        \int_{\Omega} \phi^{3p} u^{n-2p+2} |\nabla u|^{3p-6} |D^2u \nabla u|^2 \dd x
        +
        \int_{\Omega} \phi^{3p} u^{n-2p+2} |\nabla u|^{3p-4} |D^2u|^2 \dd x
        \\
        & + \int_{\partial\Omega}\phi^{3p}  u^{n-2p+2}|\nabla u|^{3p-4} \mathrm{II}(\nabla u,\nabla u) \dd \mathcal{H}^{d-1} 
        \leq\ 
        c_2\left( \int_{\Omega} \phi^{3p} u^n|\nabla \Delta u|^p \dd x + \int_{\Omega} |\nabla \phi|^{3p} u^{n+p} \dd x\right)
    \end{split}
\end{equation}
and 
\begin{equation} \label{eq:Bernis_new_III}
    \begin{split}
        \int_{\Omega} \phi^{3p}u^{n-\frac{p}{2}} |\Delta u|^{\frac{3p}{2}} \dd x \leq c_{3}\left(\int_{\Omega} \phi^{3p} u^n |\nabla \Delta u|^p \dd x + \int_{\Omega} |\nabla \phi|^{3p} u^{n+p} \dd x \right)
    \end{split}
\end{equation}
hold true for all non-negative test functions $\phi \in W^1_\infty(\Omega)$ satisfying $\int_\Omega |\nabla \varphi|^{3p} u^{n+p} \dd x < \infty$. 
Here, $\mathrm{II}$ denotes the (positive semi-definite) second fundamental form of $\partial\Omega$. 
\end{theorem}

\begin{remark}\label{rem:Bernis}
\begin{itemize}
     \item[(i)] For $p=2$ and $\varphi \equiv 1$, the inequalities in \eqref{eq:Bernis_new_I} and \eqref{eq:Bernis_new_II} recover exactly the results in \cite[Theorem 1.1]{gruen2001}. The localized version for $p=2$ with a non-negative localization function $\varphi$ is mentioned in \cite{gruen2005}, where a proof is sketched for the range $2 - \sqrt{\tfrac{8}{8+d}}< n < 3$. 
    \item[(ii)] For $d=1$, estimates \eqref{eq:Bernis_new_I} and \eqref{eq:Bernis_new_III} are derived in \cite[Proposition 4.1]{ansini2004} for $p \geq \frac{4}{3}$ and $\frac{p-1}{2} < n < 2p-1$. The restriction $p \geq 2$ in our case (instead of $p \geq \frac{4}{3}$, as in $d=1$) is due to the term $|\nabla u|^{3p-6}$ in the integral
    \begin{equation*}
        \int_\Omega \varphi^{3p} u^{n-2p+2} |\nabla u|^{3p-6} |D^2u \nabla u|^2 \dd x.
    \end{equation*}
    In the case $d=1$, this integral simplifies to  
    \begin{equation*}
        \int_\Omega \varphi^{3p} u^{n-2p+2} |\nabla u|^{3p-4} |D^2u|^2 \dd x.
    \end{equation*}
    Note, however, that we can estimate $|D^2u \nabla u|^2 \leq  |D^2 u|^2 |\nabla u|^2$, where $|D^2u|= \sqrt{\sum_{i,j=1}^d (\partial_{ij}u)^2}$ denotes the Frobenius norm of the Hessian matrix $D^2u=(\partial_{ij}u)_{i,j=1,\dots,d}$.  
    \item[(iii)] 
    The restriction $p < \frac{19}{3}$ is used only once, in \eqref{eq:proof-BernisProp23-5} of the proof of Proposition \ref{prop:Bernis_2p-2-2p-1}, to ensure the non-negativity of the integral $\mathcal{I}_4$. This integral does not appear for $d=1$, as in this case $D^2u=\Delta u$.
    \item[(iv)] For smooth functions, Theorem \ref{thm:Bernis} is valid for any dimension $d \geq 1$. The critical dimension $d < 3p$ in Theorem \ref{thm:Bernis} increases monotonically with $p$. It can be computed from the conditions $W^3_p(\Omega) \hookrightarrow W^2_q(\Omega)$ for $1\leq q < \frac{pd}{d-p}$  and $W^2_q(\Omega) \hookrightarrow C(\bar{\Omega})$ for $q > \frac{d}{2}$, as evident from the proof of Theorem \ref{thm:Bernis}.
    \item[(v)] For $p\geq  2$, we have 
    \begin{equation*}
        2p - 2 \frac{2 + \sqrt{(3p-2)^2 + d(3p-4)}}{-3p+6+\sqrt{(3p-2)^2 + d(3p-4)}} < 2p-2.
    \end{equation*}
    For $p=2$, this reduces to $2 - \sqrt{\frac{8}{8+d}} < 2$, which is consistent with \cite{gruen2001,gruen2005}.
For $d = 1$, we have  
\begin{equation*}
2p - 2 \frac{2 + \sqrt{(3p-2)^2 + d(3p-4)}}{-3p + 6 + \sqrt{(3p-2)^2 + d(3p-4)}} = \frac{4p - 2 - 2\sqrt{p(p-1)}}{3},
\end{equation*}
which is greater than $\frac{p-1}{2}$ for all $p \geq 1$. This shows that the bound for $n$ derived in this article is stricter than the one provided in \cite{ansini2004}.
\end{itemize}
\end{remark}

\begin{theorem} \label{thm:consequence}
Let $\Omega\subset \R^d$, $d < 3p$,  be a bounded, convex domain with smooth boundary $\partial\Omega$. Let $ 2\leq p < \frac{19}{3}$ and 
\begin{equation*}
    2p - 2\frac{2 + \sqrt{(3p-2)^2 + d(3p-4)}}{-3p+6 + \sqrt{(3p-2)^2 + d(3p-4)}} 
    < n 
    < 2p-1.
\end{equation*}
Assume $u\in W^2_p(\Omega)$ is  positive, satisfies $\nabla u \cdot \nu = 0$ on $\partial\Omega$, and
\begin{equation*}
    \int_\Omega u^n |\nabla \Delta u|^p\, dx < \infty.
\end{equation*}
Then, we have
\begin{equation}\label{eq:consequence}
    u^{\frac{n+p}{3p}} \in W^1_{3p}(\Omega), \quad 
    u^{\frac{2(n+p)}{3p}} \in W^2_{\frac{3p}{2}}(\Omega),\quad  \text{and} \quad u^{\frac{n+p}{p}} \in W^3_p(\Omega).
\end{equation}
Moreover, there exists a constant \(c_4=c_4(n,p,d)>0\) such that
\begin{equation}\label{eq:Bernis-inequality-more} 
    \begin{split}
        &\int_{\Omega} \Big|\nabla \big(u^{\frac{n+p}{3p}}\big)\Big|^{3p} \dd x + 
        \int_{\Omega} \Big|\Delta\big(u^{\frac{2(n+p)}{3p}}\big)\Big|^{\frac{3p}{2}} \dd x
        +
        \int_{\Omega}   \Big|\nabla \Delta\big(u^{\frac{n+p}{p}}\big)\Big|^{p} \dd x  
        \leq c_4 \int_{\Omega} u^n|\nabla \Delta u|^p \dd x.
    \end{split}
\end{equation}
\end{theorem}

\begin{remark}
For the first two terms in \eqref{eq:Bernis-inequality-more}, we also obtain localized estimates of the form 
\begin{equation*}
    \begin{split}
        &\int_{\Omega}
        \varphi^{3p} \Big|\nabla \big(u^{\frac{n+p}{3p}}\big)\Big|^{3p} \dd x + 
        \int_{\Omega}
        \varphi^{3p} \Big|\Delta\big(u^{\frac{2(n+p)}{3p}}\big)\Big|^{\frac{3p}{2}} \dd x 
        \leq 
        c(n,p,d)\left( \int_{\Omega} 
        \varphi^{3p} u^n|\nabla \Delta u|^p \dd x
        \int_{\Omega} 
        |\nabla \varphi|^{3p} u^{n+p}\dd x
        \right)
    \end{split}
\end{equation*}
for all non-negative test functions $\phi \in W^1_\infty(\Omega)$ satisfying $\int_\Omega |\nabla \phi|^{3p} u^{n+p}\, dx < \infty$, see Corollary \ref{cor:improve_weighted}. However, deriving a localized version for the third-order term in \eqref{eq:Bernis-inequality-more} turns out to be significantly more challenging. This issue is discussed in Remark \ref{rem:consequences} in detail.
\end{remark}

We emphasize again that the strategy used to prove inequalities \eqref{eq:Bernis_new_I} and \eqref{eq:Bernis_new_II} in Theorem \ref{thm:Bernis} is strongly inspired by the approach in \cite{gruen2001}, while incorporating the necessary generalizations. In the Newtonian case $p=2$, these two inequalities alone are sufficient to derive the weighted analogue of Theorem \ref{thm:consequence}. Inequality \eqref{eq:Bernis_new_III} in Theorem \ref{thm:Bernis} can be established similarly to the one-dimensional case in \cite{ansini2004}. Nonetheless, the results presented here are novel and play a significant role in studying weak solutions to \eqref{eq:PDE}, as well as in proving qualitative properties such as finite speed of propagation and the occurrence of waiting-time phenomena. Further details on potential applications of Theorems \ref{thm:Bernis} and \ref{thm:consequence} are given in Section \ref{sec:related_results}, where we summarize the state of the art for the one-dimensional non-Newtonian case and for the Newtonian case in dimensions $1 \leq d < 6$.

\medskip


\subsection{Related results} \label{sec:related_results}

In the one-dimensional Newtonian case, the corresponding thin-film equation is given by
\begin{equation}\label{eq:1d_Newtonian}
u_t + \bigl(u^n u_{xxx}\bigr)_x = 0.
\end{equation}
This equation was first studied in the seminal paper \cite{BF1990}. The property of finite speed of propagation for \eqref{eq:1d_Newtonian} was shown in \cite{bernis96_1,bernis96_2} for $0 < n < 2$ and $2 \leq n < 3$, respectively, with upper bounds on the propagation rate for  $0 < n < 2$  derived in \cite{bernis96_1}. The ideas used in the proofs differ depending on the mobility exponent $n$. For  $0 < n < 2$, the proof relies heavily on the so-called $\alpha$-entropy estimates, while for $2 \leq n < 3$, it is based on localized versions of the energy inequality \eqref{eq:energy_inequ} in $d = 1$. In \cite{HS1998}, the authors showed a finite speed of propagation result in terms of the  $L_1$-norm. The occurrence of a waiting-time phenomenon was proved in \cite{dPGG2001}, even for higher space dimensions.

The one-dimensional non-Newtonian analogue of \eqref{eq:PDE} is given by
\begin{equation*}
u_t + \bigl(u^n |u_{xxx}|^{p-2} u_{xxx}\bigr)_x = 0,
\end{equation*}
and Bernis estimates in the style of Theorems \ref{thm:Bernis} and \ref{thm:consequence} were derived in \cite{ansini2004}, where the authors used these estimates in particular to guarantee the zero-contact angle condition for the constructed global non-negative solutions.  Earlier, estimates of the form \eqref{eq:Bernis_new_I} and \eqref{eq:Bernis_new_III} were obtained in \cite{bernis96_3}, where the parameter range  $p \geq \frac{4}{3}$  and  $n \in \big(\frac{p-1}{2}, \frac{p+1}{2} - \frac{1}{3p}\big)$ was considered. 
Let us note that $\alpha$-entropy estimates are not available in the non-Newtonian setting. 

The higher-dimensional Newtonian thin-film equation reads
\begin{equation}\label{eq:PDE_Newtonian}
u_t + \diverg\bigl(u^n \nabla \Delta u\bigr) = 0,
\end{equation}
and Bernis estimates for this equation were derived in \cite{gruen2001}. An entropy estimate for \eqref{eq:PDE_Newtonian} was shown in \cite{dPGG1998}, and based on this, further qualitative results for the parameter range $0 < n < 2$ -- including finite speed of propagation, waiting-time phenomena, and the existence of solutions to \eqref{eq:PDE_Newtonian} emanating from measure-valued initial data -- were obtained in \cite{BdPGG1998}, \cite{dPGG1999}, and \cite{dPGG2001}.

To the best of the authors’ knowledge, this is the first paper to address Bernis estimates for the higher-dimensional non-Newtonian thin-film equation.

\medskip


\subsection{Notation} \label{sec:notation} 

Let $\Omega \subset \R^d$ be a domain. The standard $L_q(\Omega)$-norm for $1 \leq q \leq \infty$ is denoted by $\| \cdot \|_{L_q(\Omega)}$. The Sobolev space $W^m_q(\Omega)$, where $m \in \N$, consists of functions whose generalized derivatives up to order $m$ belong to $L_q(\Omega)$. The corresponding norm is denoted by $\| \cdot \|_{W^m_q(\Omega)}$.

For a sufficiently differentiable function $u : \Omega \to \R$, the gradient of $u$ is denoted by $\nabla u$, and the Hessian matrix of $u$ is denoted by $D^2 u$. The Laplacian $\Delta u$ is defined as the trace of the Hessian matrix $D^2 u$.

The exterior unit normal vector to the boundary $\partial \Omega$ is denoted by $\nu$. Depending on the context, $|\cdot|$ represents the absolute value for scalars in $\R$, the Euclidean norm for vectors in $\R^d$, or the Frobenius norm for matrices in $\R^{d \times d}$.

\medskip


This article is organized into three sections: in Section \ref{sec:secBernis}, we prove Theorem \ref{thm:Bernis}, and in Section \ref{sec:secconsequence}, we prove Theorem \ref{thm:consequence}.

\medskip


\section{Proof of Theorem \ref{thm:Bernis} } \label{sec:secBernis}
 This main section of the paper is devoted to the proof of Theorem \ref{thm:Bernis}, and we follow, more or less, the approach of Grün \cite{gruen2001}, where Bernis estimates for the higher-dimensional thin-film equation were established in the Newtonian setting.

The proof proceeds in three main steps. First, we derive a technical auxiliary integral identity, presented in Lemma \ref{lem:BernisLemma21}. Next, we establish the desired estimates \eqref{eq:Bernis_new_I} and \eqref{eq:Bernis_new_II} for two specific ranges of $n$. For the range $2p-2 < n < 2p-1$, the estimates are proved in Proposition \ref{prop:Bernis_2p-2-2p-1}. For the range
\begin{equation*}
    2p - 2\frac{2 + \sqrt{(3p-2)^2 + d(3p-4)}}{-3p+6 + \sqrt{(3p-2)^2 + d(3p-4)}} 
    < 
    n 
    < 
    2p - 2 \frac{6p-6 + \sqrt{(3p-2)^2 + d(3p-4)}}{9p-10+\sqrt{(3p-2)^2 + d(3p-4)}}
\end{equation*} 
the estimates are obtained in Proposition \ref{prop:Bernis_smaller-2p-2}. In both cases, these results are shown for smooth functions. The inequality \eqref{eq:Bernis_new_III} is proved directly in Proposition \ref{prop:Bernis_new_III} for the full range
\begin{equation*}
    2p - 2\frac{2 + \sqrt{(3p-2)^2 + d(3p-4)}}{-3p+6 + \sqrt{(3p-2)^2 + d(3p-4)}} 
    < 
    n 
    < 
    2p - 1.
\end{equation*}
This result is also established for smooth functions. Finally, an approximation argument is used to extend these results to $u \in W^2_p(\Omega)$. Since 
\begin{equation*}
    2p - 2\frac{2 + \sqrt{(3p-2)^2 + d(3p-4)}}{-3p+6 + \sqrt{(3p-2)^2 + d(3p-4)}} 
    <
    2p-2
    \quad \text{and} \quad
    2p - 2 \frac{6p-6 + \sqrt{(3p-2)^2 + d(3p-4)}}{9p-10+\sqrt{(3p-2)^2 + d(3p-4)}}
    <
    2p-1,
\end{equation*}
we conclude Theorem \ref{thm:Bernis} for the entire range 
\begin{equation*}
    2p - 2\frac{2 + \sqrt{(3p-2)^2 + d(3p-4)}}{-3p+6 + \sqrt{(3p-2)^2 + d(3p-4)}} 
    <
    n
    <
    2p-1.
\end{equation*}

Lemma \ref{lem:BernisLemma21}, as mentioned earlier, provides the auxiliary integral identity. To prove this lemma, we use the following relation between the normal derivative of $|\nabla u|^2$ and the second fundamental form of the boundary of $\Omega$.

\begin{lemma}\label{lem:II}
Let $\Omega \subset \R^d$ be a bounded smooth domain, and let $\mathrm{II}$ denote the second fundamental form of $\partial \Omega$. For any function $u \in C^2(\bar{\Omega})$ with $\nabla u \cdot \nu = 0$ on $\partial \Omega$, we have $\nabla |\nabla u|^2 \cdot \nu 
= -2 \mathrm{II}(\nabla u, \nabla u)$. In particular, if $\Omega$ is convex, then  $\nabla |\nabla u|^2 \cdot \nu \leq 0$ on $\partial \Omega$. 
\end{lemma}
This control of the normal derivative of  $|\nabla u|^2$  on  $\partial \Omega$  is widely used in the literature; see, for example, \cite{castenholland78}, \cite{sternbergzumbrun}, and the references cited in \cite{gruen2001}. For details, we refer to these works. 

\begin{lemma}\label{lem:BernisLemma21}
Let $\Omega \subset \R^d$ be a bounded smooth domain, and let $p \geq 2$.  For any positive function $u\in C^{\infty}(\bar{\Omega})$ satisfying $\nabla u \cdot \nu = 0$ on \(\partial \Omega\), and for any \(\phi \in W^{1}_\infty(\Omega)\), the following identity holds:
\begin{equation}\label{eq:BernisLemma21}
\begin{split}
    & \int_{\Omega} \phi^{3p} u^{n-2p+2} |\nabla u|^{3p-6} |D^2u\nabla u|^2  \dd x \\
    & = - \frac{3p}{3p-4}\int_{\Omega} \phi^{3p-1} u^{n-2p+2}|\nabla u|^{3p-4} \nabla\phi \cdot D^2u \nabla u \dd x  
    \\
    &\quad - \frac{n-2p+2}{3p-4} \int_{\Omega} \phi^{3p} u^{n-2p+1} |\nabla u|^{3p-4} \nabla u\cdot D^2u \nabla u \dd x \\ 
    &\quad - \frac{1}{3p-4} \int_{\Omega} \phi^{3p} u^{n-2p+2} |\nabla u|^{3p-4} |D^2u|^2 \dd x  - \frac{1}{3p-4} \int_{\Omega} \phi^{3p} u^{n-2p+2} |\nabla u|^{3p-4} \nabla u \cdot \nabla\Delta u \dd x \\ 
    &\quad- \frac{1}{3p-4} \int_{\partial \Omega} \phi^{3p} u^{n-2p+2} |\nabla u|^{3p-4} \mathrm{II}(\nabla u,\nabla u) \dd \mathcal{H}^{d-1}.
\end{split}
\end{equation}
\end{lemma}

\begin{proof}
We first use the identities 
\begin{equation*}
    D^2u \nabla u = \tfrac{1}{2} \nabla |\nabla u|^2
    \quad
    \text{and}
    \quad
    \nabla |\nabla u|^2 \cdot \nabla |\nabla u|^2 = \diverg\bigl(|\nabla u|^2\, \nabla |\nabla u|^2\bigr) - |\nabla u|^2 \Delta |\nabla u|^2
\end{equation*} 
to obtain
\begin{equation*}
\begin{split}
    &
    \int_{\Omega} \phi^{3p} u^{n-2p+2} |\nabla u|^{3p-6} |D^2u\nabla u|^2  \dd x \\
    &= \frac{1}{4} \int_{\Omega} \phi^{3p} u^{n-2p+2} |\nabla u|^{3p-6} \nabla |\nabla u|^2 \cdot \nabla |\nabla u|^2 \dd x \\
    &=
    \frac{1}{4} \int_{\Omega} \phi^{3p} u^{n-2p+2}
    |\nabla u|^{3p-6}
    \diverg(|\nabla u|^2\, \nabla |\nabla u|^2) \dd x
    -
    \frac{1}{4} \int_{\Omega} \phi^{3p} u^{n-2p+2}
    |\nabla u|^{3p-4}
    \Delta |\nabla u|^2 \dd x.
\end{split}
\end{equation*}
Applying integration by parts on the first integral on the right-hand side and using that $\nabla |\nabla u|^2\cdot \nu =  -2\mathrm{II}(\nabla u,\nabla u)$ on $\partial \Omega$, see Lemma \ref{lem:II}, then yields
\begin{equation*}
\begin{split}
    & 
    \int_{\Omega} \phi^{3p} u^{n-2p+2} |\nabla u|^{3p-6} |D^2u\nabla u|^2  \dd x \\
    &= - \frac{3p}{4} \int_{\Omega} \phi^{3p-1} u^{n-2p+2}|\nabla u|^{3p-4} \nabla\phi \cdot \nabla |\nabla u|^2 \dd x - \frac{n-2p+2}{4} \int_{\Omega} \phi^{3p} u^{n-2p+1} |\nabla u|^{3p-4} \nabla u \cdot \nabla |\nabla u|^2 \dd x \\
    & \quad - \frac{3p-6}{4} \int_{\Omega} \phi^{3p} u^{n-2p+2} |\nabla u|^{3p-6} D^2 u\nabla u \cdot \nabla |\nabla u|^2 \dd x - \frac{1}{4} \int_{\Omega} \phi^{3p}u^{n-2p+2} |\nabla u|^{3p-4} \Delta |\nabla u|^2 \dd x \\
    & \quad - \frac{1}{2} \int_{\partial\Omega} \phi^{3p} u^{n-2p+2} |\nabla u|^{3p-4} \mathrm{II}(\nabla u,\nabla u) \dd \mathcal{H}^{d-1}.
\end{split}
\end{equation*}
With the identities 
$\nabla |\nabla u|^2 = 2 D^2 u \nabla u$ and 
$\Delta |\nabla u|^2 = 2(|D^2u|^2 + \nabla u\cdot \nabla \Delta u)$, this can be rewritten as
\begin{equation*}
\begin{split}
    & \int_{\Omega} \phi^{3p} u^{n-2p+2}  |\nabla u|^{3p-6} |D^2u\nabla u|^2  \dd x \\
    & = - \frac{3p}{2} \int_{\Omega} \phi^{3p-1} u^{n-2p+2}|\nabla u|^{3p-4} \nabla\phi \cdot D^2u \nabla u \dd x - \frac{n-2p+2}{2} \int_{\Omega} \phi^{3p} u^{n-2p+1} |\nabla u|^{3p-4}  D^2 u \nabla u \cdot \nabla u \dd x \\
    & \quad - \frac{3p-6}{2} \int_{\Omega} \phi^{3p} u^{n-2p+2} |\nabla u|^{3p-6} |D^2 u\nabla u|^2 \dd x - \frac{1}{2} \int_{\Omega} \phi^{3p} u^{n-2p+2} |\nabla u|^{3p-4} |D^2u|^2 \dd x\\
    & \quad  - \frac{1}{2} \int_{\Omega} \phi^{3p} u^{n-2p+2} |\nabla u|^{3p-4} \nabla u \cdot \nabla \Delta u \dd x - \frac{1}{2} \int_{\partial\Omega} \phi^{3p} u^{n-2p+2} |\nabla u|^{3p-4} \mathrm{II}(\nabla u,\nabla u) \dd \mathcal{H}^{d-1}.
\end{split}
\end{equation*}
Since the third integrand on the right-hand side equals the integrand on the left-hand side and since $p \geq 2 > \tfrac{4}{3}$  
ensures that $\frac{1}{3p-4} > 0$,
the proof is complete.
\end{proof}

We remark that in Lemma \ref{lem:BernisLemma21}, we need $p \geq 2 > \frac{4}{3}$ in the final step of the proof to ensure that $\frac{2}{3p-4} > 0$. In this case, we have $\frac{n-2p+2}{3p-4} > 0$ if and only if $n > 2p-2$.
Next, we establish Theorem \ref{thm:Bernis} for fixed $2 \leq p < \frac{19}{3}$ in the case $2p-2<n<2p-1$. Note that this corresponds to the bounds $2 < n < 3$ in the Newtonian case $p=2$, as detailed in \cite[Lemma 2.3]{gruen2001}.

\begin{proposition}\label{prop:Bernis_2p-2-2p-1}
Let $\Omega\subset \R^d$ be a bounded, convex and smooth domain, and let $2 \leq p < \frac{19}{3}$ be fixed. Assume that $u\in C^{\infty}(\bar{\Omega})$ is positive and satisfies $\nabla u \cdot \nu = 0$ on \(\partial \Omega\), and let \(\phi \in W^{1}_\infty(\Omega)\) be non-negative. 
If
\begin{equation*}
    2p-2 < n < 2p-1 
\end{equation*}
then there exist constants \(C_i=C_i(n,p) > 0\), $i=1,2$, such that the following estimates hold:
\begin{equation}\label{eq:BernisLemma23-1}
    \begin{split}
        \int_{\Omega} \phi^{3p}u^{n-2p} |\nabla u|^{3p} \dd x \leq C_{1}\left[ \int_{\Omega} \phi^{3p} u^n |\nabla \Delta u|^p \dd x + \int_{\Omega} |\nabla \phi|^{3p} u^{n+p} \dd x \right]
    \end{split}
    \end{equation}
    and
    \begin{equation}\label{eq:BernisLemma23-2}
    \begin{split}
       & \int_{\Omega} \phi^{3p} u^{n-2p+2} |\nabla u|^{3p-6} |D^2u\nabla u|^2 \dd x + \int_{\Omega} \phi^{3p} u^{n-2p+2}|\nabla u|^{3p-4} |D^2u|^2 \dd x\\
       & + \int_{\Omega}\phi^{3p} u^{n-2p+2}|\nabla u|^{3p-4} |\Delta u|^2 \dd x  + \int_{\partial \Omega}\phi^{3p} u^{n-2p+2}|\nabla u|^{3p-4} \mathrm{II}(\nabla u,\nabla u) \dd \mathcal{H}^{d-1} \\
       \leq\ & C_{2}\left[ \int_{\Omega}\phi^{3p}  u^n |\nabla \Delta u|^p \dd x + \int_{\Omega} |\nabla \phi|^{3p} u^{n+p} \dd x \right].
    \end{split}
    \end{equation}
\end{proposition}
\smallskip

Note that Proposition \ref{prop:Bernis_2p-2-2p-1} works for any dimension $d \geq 1$. 

\begin{remark}
The constants $C_1$ and $C_2$ in Proposition \ref{prop:Bernis_2p-2-2p-1} are given by
\begin{equation*}
\begin{split}
    C_1
    &=
    \max\left\{
    \frac{3p}{2} \left|\frac{2(3p-1)}{(n-2p+1)(n-2p+2)}\right|^p,
    \right.
    \\
    & \quad
    \left.
    p^\frac{3p}{2} \left(\frac{9p^2}{2(n-2p+1)^2}+\frac{9p^2 (3p-2)^2}{\eps_2(n-2p+1)^2 (n-2p+2)^2}
    +
    \frac{9p^2}{\eps_3(n-2p+1)^2 (n-2p+2)^2}
    \right)^\frac{3p}{2}
    \right\}
\end{split}
\end{equation*}
and
\begin{equation*}
    C_2
    =
    \frac{5}{3p}
    +
    C_1 \left[
    \frac{p-1}{p} \left(\frac{3p-1}{C_{\min}}\right)^\frac{p}{p-1}
    +
    \frac{3p-2}{3p}
    \left(\frac{C_{\delta_0} + C_{\delta_1} + C_{\delta_2}}{C_{\min}}\right)^\frac{3p}{3p-2}
    +
    \frac{1}{C_{\min}}
    \right],
\end{equation*}
where 
\begin{equation*}
    C_{\delta_0} = \frac{9 p^2 (n-2p+2)^2}{4} > 0,
\end{equation*}
and $C_{\min}, C_{\delta_1}$, and $C_{\delta_2}$ are positive constants depending only on the fixed parameter $p$, but not on $n$.
Hence, both $C_1$ and $C_2$ blow up as $n$ goes to $2p-2$. 
Moreover, we need  $n > 2p-2$  in the proof to guarantee that the involved integrals have the `right’ sign. However, it turns out that the lower bound $2p-2$ can be slightly improved (see Proposition \ref{prop:Bernis_smaller-2p-2} below).
\end{remark}
\smallskip

\begin{proof}
We first prove \eqref{eq:BernisLemma23-1}. 
To this end, we use the identity $\nabla u \cdot \nabla u = \diverg(u \nabla u) - u \Delta u$
and integration by parts, together with the boundary condition for $u$, to rewrite the left-hand side of \eqref{eq:BernisLemma23-1} as 
\begin{equation*}
\begin{split}
    \int_{\Omega} \phi^{3p} u^{n-2p} |\nabla u|^{3p} \dd x 
    & = \int_{\Omega} \phi^{3p}  u^{n-2p} |\nabla u|^{3p-2} \nabla u \cdot \nabla u \dd x \\
    &=
    \int_{\Omega} \phi^{3p}  u^{n-2p} |\nabla u|^{3p-2} \diverg(u \nabla u) \dd x
    -
    \int_{\Omega} \phi^{3p}  u^{n-2p+1} |\nabla u|^{3p-2} \Delta u \dd x
    \\
    & = -3p\int_{\Omega} \phi^{3p-1} u^{n-2p+1}|\nabla u|^{3p-2} \nabla \phi \cdot \nabla u \dd x  - (n-2p) \int_{\Omega}\phi^{3p} u^{n-2p} |\nabla u|^{3p} \dd x \\
    & \quad - (3p-2) \int_{\Omega}\phi^{3p} u^{n-2p+1} |\nabla u|^{3p-4}  D^2 u\nabla u \cdot \nabla u \dd x -\int_{\Omega}\phi^{3p} u^{n-2p+1} |\nabla u|^{3p-2} \Delta u \dd x.
\end{split}
\end{equation*}
Since the second integrand on the right-hand side equals the integrand on the left-hand side and since $n < 2p-1$ by assumption, we find that
\begin{equation}\label{eq:proof-BernisProp23-1.}
    \begin{split}
        \int_{\Omega} \phi^{3p} u^{n-2p} |\nabla u|^{3p} \dd x & = - \frac{3p}{n-2p+1} \int_{\Omega}\phi^{3p-1} u^{n-2p+1}|\nabla u|^{3p-2} \nabla \phi \cdot \nabla u \dd x \\
        & \quad - \frac{3p-2}{n-2p+1} \int_{\Omega} \phi^{3p} u^{n-2p+1} |\nabla u|^{3p-4}  D^2 u\nabla u \cdot \nabla u \dd x \\
        & \quad - \frac{1}{n-2p+1}\int_{\Omega} \phi^{3p} u^{n-2p+1} |\nabla u|^{3p-2} \Delta u \dd x.
    \end{split}
\end{equation}
In the last term of \eqref{eq:proof-BernisProp23-1.} we use again the identity $\nabla u \cdot \nabla u = \diverg(u \nabla u) - u \Delta u$ and integrate by parts to obtain 
\begin{equation*}
    \begin{split}
        &\int_{\Omega} \phi^{3p} u^{n-2p+1} |\nabla u|^{3p-2} \Delta u \dd x\\
        =\ &
        \int_\Omega \phi^{3p}  u^{n-2p+1} |\nabla u|^{3p-4} \diverg(u \nabla u) \Delta u \dd x
        -
        \int_\Omega \phi^{3p}  u^{n-2p+2} |\nabla u|^{3p-4} |\Delta u|^2 \dd x
        \\
        =\ & 
        -3p \int_{\Omega}  \phi^{3p-1} u^{n-2p+2} |\nabla u|^{3p-4} \Delta u \nabla \phi \cdot \nabla u \dd x 
        -
        \int_{\Omega} \phi^{3p} u^{n-2p+2} |\nabla u|^{3p-4} |\Delta u|^2 \dd x\\
        &
        - (n-2p+1) 
        \int_{\Omega} \phi^{3p} u^{n-2p+1} |\nabla u|^{3p-2} \Delta u \dd x 
        - (3p-4) \int_{\Omega} \phi^{3p} u^{n-2p+2} |\nabla u|^{3p-6} \Delta u D^2 u \nabla u\cdot \nabla u \dd x \\
        & - 
        \int_{\Omega} \phi^{3p} u^{n-2p+2} |\nabla u|^{3p-4} \nabla \Delta u \cdot \nabla u \dd x.   
    \end{split}
\end{equation*}
Noting that the integrand on the left-hand side equals the third integrand on the right-hand side and recalling that $n > 2p-2$ by assumption, we deduce 
\begin{equation*}
    \begin{split}
        & 
        \int_{\Omega} \phi^{3p} u^{n-2p+2} |\nabla u|^{3p-2} \Delta u \dd x \\
        =\ &
        -\frac{3p}{n-2p+2} \int_{\Omega}  \phi^{3p-1} u^{n-2p+2} |\nabla u|^{3p-4} \Delta u \nabla \phi \cdot \nabla u \dd x 
        - 
        \frac{1}{n-2p+2} \int_{\Omega} \phi^{3p} u^{n-2p+2} |\nabla u|^{3p-4} |\Delta u|^2 \dd x  \\
        & - \frac{3p-4}{n-2p+2} \int_{\Omega} \phi^{3p} u^{n-2p+2} |\nabla u|^{3p-6} \Delta u D^2 u \nabla u\cdot \nabla u \dd x  
        \\
        &
        - \frac{1}{n-2p+2} \int_{\Omega} \phi^{3p} u^{n-2p+2} |\nabla u|^{3p-4} \nabla \Delta u \cdot \nabla u \dd x.
    \end{split}
\end{equation*}
Inserting this into \eqref{eq:proof-BernisProp23-1.}, we obtain 
\begin{equation}\label{eq:proof-BernisProp23-2}
    \begin{split}
        \int_{\Omega}\phi^{3p} u^{n-2p} |\nabla u|^{3p} \dd x & = - \frac{3p}{n-2p+1} \int_{\Omega}\phi^{3p-1} u^{n-2p+1}|\nabla u|^{3p-2} \nabla \phi \cdot \nabla u \dd x\\
        & \quad - \frac{3p-2}{n-2p+1} \int_{\Omega} \phi^{3p} u^{n-2p+1} |\nabla u|^{3p-4}  D^2 u\nabla u \cdot \nabla u \dd x \\
        & \quad + \frac{3p}{(n-2p+1)(n-2p+2)} \int_{\Omega}  \phi^{3p-1} u^{n-2p+2} |\nabla u|^{3p-4} \Delta u \nabla \phi \cdot \nabla u \dd x \\
        & \quad + \frac{1}{(n-2p+1)(n-2p+2)} \int_{\Omega} \phi^{3p} u^{n-2p+2} |\nabla u|^{3p-4} |\Delta u|^2 \dd x  \\
        & \quad + \frac{3p-4}{(n-2p+1)(n-2p+2)} \int_{\Omega} \phi^{3p} u^{n-2p+2} |\nabla u|^{3p-6} \Delta u D^2 u \nabla u\cdot \nabla u \dd x \\
        & \quad + \frac{1}{(n-2p+1)(n-2p+2)} \int_{\Omega} \phi^{3p} u^{n-2p+2} |\nabla u|^{3p-4} \nabla \Delta u \cdot \nabla u \dd x.
    \end{split}
\end{equation}
Next, we use Lemma \ref{lem:BernisLemma21} to rewrite the second term on the right-hand side of \eqref{eq:proof-BernisProp23-2} as follows
\begin{equation}\label{eq:proof-BernisProp23-3}
    \begin{split}
        & - \frac{3p-2}{n-2p+1} \int_{\Omega} \phi^{3p} u^{n-2p+1} |\nabla u|^{3p-4} D^2u\nabla u\cdot \nabla u \dd x \\
        =\ & \frac{(3p-2)(3p-4)}{(n-2p+1)(n-2p+2)} \int_{\Omega} \phi^{3p} u^{n-2p+2} |\nabla u|^{3p-6} |D^2u\nabla u|^2  \dd x \\
        & 
        +  \frac{3p(3p-2)}{(n-2p+1)(n-2p+2)}
        \int_{\Omega} \phi^{3p-1} u^{n-2p+2}|\nabla u|^{3p-4} \nabla\phi \cdot D^2u \nabla u \dd x  \\ 
        & 
        + \frac{3p-2}{(n-2p+1)(n-2p+2)} \int_{\Omega} \phi^{3p} u^{n-2p+2} |\nabla u|^{3p-4} |D^2u|^2 \dd x  \\
        & + \frac{3p-2}{(n-2p+1)(n-2p+2)} \int_{\Omega} \phi^{3p} u^{n-2p+2} |\nabla u|^{3p-4} \nabla u \cdot \nabla\Delta u \dd x \\ 
        & + \frac{3p-2}{(n-2p+1)(n-2p+2)} \int_{\partial \Omega} \phi^{3p} u^{n-2p+2} |\nabla u|^{3p-4} \mathrm{II}(\nabla u,\nabla u) \dd \mathcal{H}^{d-1}.
    \end{split}
\end{equation}
Applying now the Cauchy--Schwarz inequality and Young's inequality with weight $\varepsilon_0 =\tfrac{1}{15}$ on the fifth term on the right-hand side of \eqref{eq:proof-BernisProp23-2} we obtain 
\begin{equation}\label{eq:proof-BernisProp23-4} 
\begin{split}
    &
    -\int_{\Omega} \phi^{3p} u^{n-2p+2} |\nabla u|^{3p-6} \Delta u D^2 u \nabla u\cdot \nabla u \dd x
    \\
    \leq\ &
    \left|
    \int_{\Omega} \phi^{3p} u^{n-2p+2} |\nabla u|^{3p-6} \Delta u D^2 u \nabla u\cdot \nabla u \dd x
    \right|
    \\
    \leq\ & 
    \frac{1}{15} \int_{\Omega} \phi^{3p} u^{n-2p+2} |\nabla u|^{3p-4} |\Delta u|^2 \dd x 
    + \frac{15}{4}
    \int_{\Omega} \phi^{3p} u^{n-2p+2} |\nabla u|^{3p-6} |D^2u\nabla u|^2 \dd x.
\end{split}
\end{equation} 
Inserting \eqref{eq:proof-BernisProp23-3} back into \eqref{eq:proof-BernisProp23-2} and using \eqref{eq:proof-BernisProp23-4} yields 
\begin{equation}\label{eq:proof-BernisProp23-5}
    \begin{split}
        & \int_{\Omega} \phi^{3p} u^{n-2p} |\nabla u|^{3p} \dd x \\
        \leq\ & 
        \frac{(3p-4)(12p-23)}{4(n-2p+1)(n-2p+2)}
        \int_{\Omega} \phi^{3p} u^{n-2p+2} |\nabla u|^{3p-6} |D^2u\nabla u|^2 \dd x \\
        & + \frac{3p-2}{(n-2p+1)(n-2p+2)} \int_{\Omega} \phi^{3p} u^{n-2p+2} |\nabla u|^{3p-4} |D^2u|^2 \dd x \\
        & + 
        \frac{3p-1}{(n-2p+1)(n-2p+2)}
        \int_{\Omega} \phi^{3p} u^{n-2p+2} |\nabla u|^{3p-4} \nabla u \cdot \nabla \Delta u \dd x \\
        & + 
        \frac{19-3p}{15(n-2p+1)(n-2p+2)}
        \int_{\Omega} \phi^{3p} u^{n-2p+2} |\nabla u|^{3p-4} |\Delta u|^2 \dd x \\
        & + \frac{3p-2}{(n-2p+1)(n-2p+2)} \int_{\partial\Omega} \phi^{3p} u^{n-2p+2} |\nabla u|^{3p-4} \mathrm{II}(\nabla u, \nabla u) \dd \mathcal{H}^{d-1} \\
        & - \frac{3p}{n-2p+1} \int_{\Omega}\phi^{3p-1} u^{n-2p+1}|\nabla u|^{3p-2} \nabla \phi \cdot \nabla u \dd x \\
        & + \frac{3p(3p-2)}{(n-2p+1)(n-2p+2)} \int_{\Omega} \phi^{3p-1} u^{n-2p+2}|\nabla u|^{3p-4} \nabla\phi \cdot D^2u \nabla u \dd x \\
        & + \frac{3p}{(n-2p+1)(n-2p+2)} \int_{\Omega}  \phi^{3p-1} u^{n-2p+2} |\nabla u|^{3p-4} \Delta u \nabla \phi \cdot \nabla u \dd x
        \\
        \eqqcolon & \
        \mathcal{I}_1
        +
        \mathcal{I}_2
        +
        \mathcal{I}_3
        +
        \mathcal{I}_4
        +
        \mathcal{I}_5
        +
        \mathcal{I}_6
        +
        \mathcal{I}_7
        +
        \mathcal{I}_8.
    \end{split}
\end{equation}
Due to the choices of $p \geq 2$ and $2p-2 < n < 2p-1$ the pre factors of $\mathcal{I}_1, \mathcal{I}_2$ and $\mathcal{I}_5$ are negative, while the integrals are non-negative. The non-negativity of \(\mathrm{II}(\nabla u,\nabla u)\) follows from the convexity of the domain, see Lemma \ref{lem:II}. Moreover, the choices of $p < \frac{19}{3}$ and $2p-2 < n < 2p-1$ yield $\mathcal{I}_4 \leq 0$. We directly use $\mathcal{I}_2, \mathcal{I}_5 \leq 0$ in \eqref{eq:proof-BernisProp23-5}.
To control the last three integrals $\mathcal{I}_6, \mathcal{I}_7$ and $\mathcal{I}_8$, we use the Cauchy--Schwarz inequality and Young's inequality such that $\mathcal{I}_6$ can be partly absorbed by the integral on the left-hand side of \eqref{eq:proof-BernisProp23-5} and the integrals $\mathcal{I}_7$ and $\mathcal{I}_8$ are absorbed by $\mathcal{I}_1$, respectively $\mathcal{I}_4$. Indeed,
applying the Cauchy-Schwarz inequality and Young's inequality with weight $\varepsilon_1 = \frac{1}{2}$ to $\mathcal{I}_6$, we have
\begin{equation}\label{eq:eps_1}
\begin{split}
    &
    - \frac{3p}{n-2p+1} \int_{\Omega}\phi^{3p-1} u^{n-2p+1}|\nabla u|^{3p-2} \nabla \phi \cdot \nabla u \dd x 
    \\
    \leq\ & 
    \frac{1}{2} \int_{\Omega} \phi^{3p} u^{n-2p}|\nabla u|^{3p} \dd x  
    + C_{\eps_1}
     \int_{\Omega}  \phi^{3p-2} |\nabla \phi|^2 u^{n-2p+2} |\nabla u|^{3p-2} \dd x
\end{split}
\end{equation}
with $C_{\varepsilon_1}=C(\varepsilon_1,n,p):=\tfrac{9p^2}{2(n-2p+1)^2}>0$.
Inserting this inequality in \eqref{eq:proof-BernisProp23-5}, we
obtain 
\begin{equation}\label{eq:proof-BernisProp23-5.}
    \begin{split}
        \int_{\Omega} \phi^{3p} u^{n-2p} |\nabla u|^{3p} \dd x 
        \leq
        2(\mathcal{I}_1
        +
        \mathcal{I}_3
        +
        \mathcal{I}_4
        +
        \mathcal{I}_7
        +
        \mathcal{I}_8) 
        + 
        2 C_{\eps_1}
        \int_{\Omega}  \phi^{3p-2} |\nabla \phi|^2 u^{n-2p+2} |\nabla u|^{3p-2} \dd x.
    \end{split}
\end{equation}
We can further estimate $\mathcal{I}_7$ by 
\begin{equation}
\label{eq:eps_2}
\begin{split}
    &
    \left|\frac{6p(3p-2)}{(n-2p+1)(n-2p+2)}\int_{\Omega} \phi^{3p-1} u^{n-2p+2}|\nabla u|^{3p-4} \nabla\phi \cdot D^2u \nabla u \dd x \right|
    \\ 
    \leq\ &  
    \eps_2 \int_{\Omega} \phi^{3p} u^{n-2p+2}|\nabla u|^{3p-6} |D^2u\nabla u|^2 \dd x 
    + C_{\eps_2} \int_{\Omega} \phi^{3p-2} |\nabla \phi|^2 u^{n-2p+2} |\nabla u|^{3p-2} \dd x,
\end{split}
\end{equation}
as well as $\mathcal{I}_8$ by 
\begin{equation}\label{eq:eps_3}
\begin{split} 
    & \left| \frac{6p}{(n-2p+1)(n-2p+2)} \int_{\Omega}  \phi^{3p-1} u^{n-2p+2} |\nabla u|^{3p-4} \Delta u \nabla \phi \cdot \nabla u \dd x \right|
    \\
    \leq\ & 
    \eps_3 \int_{\Omega} \phi^{3p} u^{n-2p+2} |\nabla u|^{3p-4} |\Delta u|^2 \dd x  + C_{\eps_3} \int_{\Omega}\phi^{3p-2} |\nabla \phi|^2 u^{n-2p+2}|\nabla u|^{3p-2} \dd x,
    \end{split}
\end{equation}
where $\eps_2, \eps_3 > 0$ can be chosen arbitrarily and where $C_{\eps_2}, C_{\eps_3} > 0$ are given by 
\[
C_{\eps_2}=C(\eps_2,n,p)
\coloneqq \frac{9p^2(3p-2)^2}{\eps_2(n-2p+1)^2(n-2p+2)^2}
\]
and
\[
C_{\eps_3}=C(\eps_3,n,p)
\coloneqq \frac{9p^2}{\eps_3(n-2p+1)^2(n-2p+2)^2}.
\]
We choose 
$\eps_2 = \eps_2(n,p)$ and $\eps_3 = \eps_3(n,p)$ small enough so that 
\begin{equation*}
    0 < \eps_2 < \left|\frac{(3p-4)(12p-23)}{2(n-2p+1)(n-2p+2)}\right| 
    \quad \text{and}\quad 
    0 < \eps_3 < \left|\frac{2(19-3p)}{15(n-2p+1)(n-2p+2)}\right|.
\end{equation*}
Using \eqref{eq:eps_2} and \eqref{eq:eps_3} in \eqref{eq:proof-BernisProp23-5.}, we can absorb the terms in \eqref{eq:proof-BernisProp23-5.} to get the estimate
\begin{equation}\label{eq:younghoelder}
    \begin{split}
        \int_{\Omega} \phi^{3p} u^{n-2p} |\nabla u|^{3p} \dd x 
        &
        \leq 
        \frac{2(3p-1)}{(n-2p+1)(n-2p+2)} \int_{\Omega} \phi^{3p}  u^{n-2p+2} |\nabla u|^{3p-4} \nabla u \cdot \nabla \Delta u \dd x 
        \\
        & \quad
        + 
        \bigl(2C_{\eps_1} + C_{\eps_2} + C_{\eps_3}\bigr)  \int_{\Omega} \phi^{3p-2}|\nabla \phi|^2 u^{n-2p+2} |\nabla u|^{3p-2} \dd x.
    \end{split}
\end{equation}
Applying Hölder's inequality to both integrals on the right-hand side of \eqref{eq:younghoelder}, we obtain
\begin{equation}\label{eq:hoelder}
    \begin{split}
    &\int_{\Omega} \phi^{3p} u^{n-2p} |\nabla u|^{3p} \dd x \\
    \leq\ & 
    \left| \frac{2(3p-1)}{(n-2p+1)(n-2p+2)} \right| \left(\int_{\Omega} \phi^{3p}  u^{n-2p} |\nabla u|^{3p} \dd x \right)^{\frac{p-1}{p}} \left(\int_{\Omega}\phi^{3p}  u^{n} |\nabla\Delta u|^{p}\dd x\right)^{\frac{1}{p}} \\
    & + \bigl(2C_{\eps_1} + C_{\eps_2} + C_{\eps_3}\bigr) \left(\int_{\Omega} \phi^{3p} u^{n-2p}|\nabla u|^{3p} \dd x  \right)^{\frac{3p-2}{3p}}\left(\int_{\Omega} |\nabla \phi|^{3p} u^{n+p}\dd x\right)^{\frac{2}{3p}}.
    \end{split}
\end{equation}
We then use Young's inequality 
on the first line on the right-hand side of \eqref{eq:hoelder} to deduce 
\begin{equation*}
    \begin{split}
    &\int_{\Omega} \phi^{3p} u^{n-2p} |\nabla u|^{3p} \dd x 
    \\
    \leq\ & \frac{p-1}{p} \int_{\Omega} \phi^{3p} u^{n-2p} |\nabla u|^{3p} \dd x  + \frac{1}{p} \left| \frac{2(3p-1)}{(n-2p+1)(n-2p+2)} \right|^p  \int_{\Omega} \phi^{3p}  u^{n} |\nabla \Delta u|^{p} \dd x 
     \\
     & +  \bigl(2C_{\eps_1} + C_{\eps_2} + C_{\eps_3}\bigr) \left(\int_{\Omega} \phi^{3p} u^{n-2p}|\nabla u|^{3p} \dd x  \right)^{\frac{3p-2}{3p}}\left(\int_{\Omega} |\nabla \phi|^{3p} u^{n+p}\dd x\right)^{\frac{2}{3p}},
    \end{split}
\end{equation*}
which immediately gives  
\begin{equation*}
    \begin{split}
    \int_{\Omega} \phi^{3p} u^{n-2p} |\nabla u|^{3p} \dd x 
     &\leq  \left| \frac{2(3p-1)}{(n-2p+1)(n-2p+2)} \right|^p  \int_{\Omega} \phi^{3p}  u^{n} |\nabla \Delta u|^{p} \dd x 
     \\
     &\quad + p \bigl(2C_{\eps_1} + C_{\eps_2} + C_{\eps_3}\bigr) \left(\int_{\Omega} \phi^{3p} u^{n-2p}|\nabla u|^{3p} \dd x  \right)^{\frac{3p-2}{3p}}\left(\int_{\Omega} |\nabla \phi|^{3p} u^{n+p}\dd x\right)^{\frac{2}{3p}}.
    \end{split}
    \end{equation*}
Applying again Young's inequality,
but now to the second line on the right-hand side of the above inequality, leads to 
\begin{equation*}
    \begin{split}
    \int_{\Omega} \phi^{3p} u^{n-2p} |\nabla u|^{3p} \dd x 
     &\leq \frac{3p}{2} \left| \frac{2(3p-1)}{(n-2p+1)(n-2p+2)} \right|^p  \int_{\Omega} \phi^{3p}  u^{n} |\nabla \Delta u|^{p} \dd x 
     \\
     &\quad + p^{\frac{3p}{2}} \bigl(2C_{\eps_1} + C_{\eps_2} + C_{\eps_3}\bigr)^{\frac{3p}{2}}\int_{\Omega} |\nabla \phi|^{3p} u^{n+p}\dd x,
    \end{split}
\end{equation*}
and from this we conclude estimate \eqref{eq:BernisLemma23-1} with the constant $C_1 > 0$ given by
\begin{equation*}
\begin{split}
    C_1
    &=
    \max\left\{
    \frac{3p}{2} \left|\frac{2(3p-1)}{(n-2p+1)(n-2p+2)}\right|^p,
    \right.
    \\
    & \quad
    \left.
    p^\frac{3p}{2} \left(\frac{9p^2}{(n-2p+1)^2}
    +
    \frac{9p^2 (3p-2)^2}{\eps_2(n-2p+1)^2 (n-2p+2)^2}
    +
    \frac{9p^2}{\eps_3(n-2p+1)^2 (n-2p+2)^2}
    \right)^\frac{3p}{2}
    \right\}.
\end{split}
\end{equation*}
It remains to prove \eqref{eq:BernisLemma23-2}.
Noting that the left-hand side in \eqref{eq:proof-BernisProp23-5} is non-negative and multiplying \eqref{eq:proof-BernisProp23-5}  by $(n-2p+1)(n-2p+2)<0$ gives 
$$ 
(n-2p+1)(n-2p+2)\sum\limits_{i=1}^8 \mathcal{I}_i \leq 0
$$ 
and by rearranging this inequality, we see that 
\begin{align}\label{eq:proof-BernisProp23-6}
    & \frac{(3p-4)(12p-23)}{4} \int_{\Omega} \phi^{3p}u^{n-2p+2} |\nabla u|^{3p-6} |D^2u\nabla u|^2 \dd x  + (3p-2)\int_{\Omega} \phi^{3p} u^{n-2p+2} |\nabla u|^{3p-4} |D^2u|^2 \dd x \nonumber\\
    & + \frac{19-3p}{15} \int_{\Omega} \phi^{3p} u^{n-2p+2} |\nabla u|^{3p-4} |\Delta u|^2 \dd x + (3p-2) \int_{\partial\Omega} \phi^{3p} u^{n-2p+2} |\nabla u|^{3p-4} \mathrm{II}(\nabla u, \nabla u) \dd \mathcal{H}^{d-1} \nonumber\\
    \leq\ & - (3p-1) \int_{\Omega} \phi^{3p} u^{n-2p+2} |\nabla u|^{3p-4} \nabla u\cdot \nabla \Delta u\dd x \nonumber \\
    & + 3p(n-2p+2) \int_{\Omega}\phi^{3p-1} u^{n-2p+1}|\nabla u|^{3p-2} \nabla \phi \cdot \nabla u \dd x \\
    & - 3p(3p-2)\int_{\Omega} \phi^{3p-1} u^{n-2p+2}|\nabla u|^{3p-4} \nabla\phi \cdot D^2u \nabla u \dd x \nonumber\\
    & - 3p\int_{\Omega}  \phi^{3p-1} u^{n-2p+2} |\nabla u|^{3p-4} \Delta u \nabla \phi \cdot \nabla u \dd x \nonumber 
    \\
    \eqqcolon &\
    \mathcal{J}_1
    +
    \mathcal{J}_2
    +
    \mathcal{J}_3
    +
    \mathcal{J}_4. \nonumber
\end{align}
Now, we estimate the terms $\mathcal{J}_2$, $\mathcal{J}_3$, and $\mathcal{J}_4$ by similar arguments as those used for \eqref{eq:eps_1}, \eqref{eq:eps_2} and \eqref{eq:eps_3}. Indeed, to estimate $\mathcal{J}_2$, we use Young's inequality with $\delta_0=1$ and \eqref{eq:BernisLemma23-1}. We obtain
\begin{equation}\label{eq:delta_0}
\mathcal{J}_2 \leq C_{\delta_0} \int_{\Omega} \phi^{3p-2} |\nabla \phi|^2 u^{n-2p+2} |\nabla u|^{3p-2} \dd x + C_1 \int_{\Omega} \phi^{3p} u^n |\nabla \Delta u|^p \dd x
+ C_1 \int_{\Omega} |\nabla \phi|^{3p} u^{n+p} \dd x
\end{equation}
with $C_{\delta_0}=C(\delta_0,n,p) \coloneqq \tfrac{9p^2(n-2p+2)^2}{4} >0$. 
For $\mathcal{J}_3$ and $\mathcal{J}_4$, we obtain the estimates 
\begin{equation}\label{eq:delta_1}
    \mathcal{J}_3 \leq \delta_1 \int_{\Omega} \phi^{3p}u^{n-2p+2} |\nabla u|^{3p-6} |D^2u\nabla u|^2 \dd x + C_{\delta_1} \int_{\Omega} \phi^{3p-2} |\nabla \phi|^2 u^{n-2p+2} |\nabla u|^{3p-2} \dd x
\end{equation}
and 
\begin{equation}\label{eq:delta_2}
    \mathcal{J}_4 \leq \delta_2 \int_{\Omega} \phi^{3p}u^{n-2p+2} |\nabla u|^{3p-4} |\Delta u|^2 \dd x + C_{\delta_2} \int_{\Omega} \phi^{3p-2} |\nabla \phi|^2 u^{n-2p+2} |\nabla u|^{3p-2} \dd x,
\end{equation}
where $\delta_1,\delta_2>0$ can be chosen arbitrarily 
and with 
$$
C_{\delta_1}=C(\delta_1,p):= \frac{9p^2(3p-2)^2}{4\delta_1}>0 \quad \text{and} \quad C_{\delta_2}=C(\delta_2,p):= \frac{9p^2}{4\delta_2}>0. 
$$
Inserting \eqref{eq:delta_0}--\eqref{eq:delta_2} into \eqref{eq:proof-BernisProp23-6} and choosing 
\begin{equation*}
    0<\delta_1< \frac{(3p-4)(12p-23)}{8}
    \quad \text{and} \quad 0<\delta_2<\frac{19-3p}{30}
\end{equation*}
yields 
\begin{align}\label{eq:proof-BernisProp23-7}
    & \frac{(3p-4)(12p-23)}{8} \int_{\Omega} \phi^{3p}u^{n-2p+2} |\nabla u|^{3p-6} |D^2u\nabla u|^2 \dd x  + (3p-2)\int_{\Omega} \phi^{3p} u^{n-2p+2} |\nabla u|^{3p-4} |D^2u|^2 \dd x \nonumber \\
    & + \frac{19-3p}{30} \int_{\Omega} \phi^{3p} u^{n-2p+2} |\nabla u|^{3p-4} |\Delta u|^2 \dd x + (3p-2) \int_{\partial\Omega} \phi^{3p} u^{n-2p+2} |\nabla u|^{3p-4} \mathrm{II}(\nabla u, \nabla u) \dd \mathcal{H}^{d-1} \nonumber \\
    \leq\ & - (3p-1) \int_{\Omega} \phi^{3p} u^{n-2p+2} |\nabla u|^{3p-4} \nabla u\cdot \nabla \Delta u\dd x \\
    & + (C_{\delta_0}+C_{\delta_1}+C_{\delta_2}) \int_{\Omega} \phi^{3p-2}
    |\nabla \phi|^2 u^{n-2p+2} |\nabla u|^{3p-2} \dd x \nonumber \\
    & + C_1 \int_{\Omega} \phi^{3p} u^n |\nabla \Delta u|^p \dd x +  C_1\int_{\Omega} |\nabla \phi|^{3p} u^{n+p} \dd x.  \nonumber
\end{align}
Setting 
$$
C_{\min} \coloneqq \min \left\{ \frac{(3p-4)(12p-23)}{8}, \frac{19-3p}{30}, 3p-2\right\}
>0
$$
and, applying Hölder's and Young's inequalities, further yields
    \begin{align*}
        & \int_{\Omega} \phi^{3p}u^{n-2p+2} |\nabla u|^{3p-6} |D^2u\nabla u|^2 \dd x  + \int_{\Omega} \phi^{3p} u^{n-2p+2} |\nabla u|^{3p-4} |D^2u|^2 \dd x \nonumber \\
        & +  \int_{\Omega} \phi^{3p} u^{n-2p+2} |\nabla u|^{3p-4} |\Delta u|^2 \dd x +  \int_{\partial\Omega} \phi^{3p} u^{n-2p+2} |\nabla u|^{3p-4} \mathrm{II}(\nabla u, \nabla u) \dd \mathcal{H}^{d-1} \nonumber \\
        \leq\ & 
        \frac{(3p-1)}{C_{\min}}
        \left(\int_{\Omega} \phi^{3p}  u^{n-2p} |\nabla u|^{3p} \dd x \right)^{\frac{p-1}{p}} \left(\int_{\Omega}\phi^{3p}  u^{n} |\nabla\Delta u|^{p}\dd x\right)^{\frac{1}{p}} \\
        & + \frac{\bigl(C_{\delta_0} + C_{\delta_1} + C_{\delta_2}\bigr)}{C_{\min}} \left(\int_{\Omega} \phi^{3p} u^{n-2p}|\nabla u|^{3p} \dd x  \right)^{\frac{3p-2}{3p}}\left(\int_{\Omega} |\nabla \phi|^{3p} u^{n+p}\dd x\right)^{\frac{2}{3p}} \\
        & + \frac{C_1}{C_{\min}} \int_{\Omega} \phi^{3p} u^n |\nabla \Delta u|^p \dd x +  \frac{C_1}{C_{\min}}\int_{\Omega} |\nabla \phi|^{3p} u^{n+p} \dd x  \nonumber
        \\
        \leq\ &
        \frac{p-1}{p}
        \left(\frac{3p-1}{C_{\min}}\right)^\frac{p}{p-1}
        \int_{\Omega} \phi^{3p}  u^{n-2p} |\nabla u|^{3p} \dd x 
        +
        \frac{1}{p}
        \int_{\Omega}\phi^{3p}  u^{n} |\nabla\Delta u|^{p}\dd x
        \\
        & +
        \frac{3p-2}{3p}
        \left(\frac{C_{\delta_0} + C_{\delta_1} + C_{\delta_2}}{C_{\min}}\right)^\frac{3p}{3p-2}
        \int_{\Omega} \phi^{3p} u^{n-2p}|\nabla u|^{3p} \dd x  
        +
        \frac{2}{3p} \int_{\Omega} |\nabla \phi|^{3p} u^{n+p}\dd x
        \\
        & + \frac{C_1}{C_{\min}} \int_{\Omega} \phi^{3p} u^n |\nabla \Delta u|^p \dd x +  \frac{C_1}{C_{\min}}\int_{\Omega} |\nabla \phi|^{3p} u^{n+p} \dd x.
\end{align*}
Using \eqref{eq:BernisLemma23-1}, we finally obtain the estimate \eqref{eq:BernisLemma23-2} with the constant $C_2 > 0$ given by
\begin{equation*}
    C_2
    \coloneqq
    \frac{5}{3p}
    +
    C_1 \left[
    \frac{p-1}{p} \left(\frac{3p-1}{C_{\min}}\right)^\frac{p}{p-1}
    +
    \frac{3p-2}{3p}
    \left(\frac{C_{\delta_0} + C_{\delta_1} + C_{\delta_2}}{C_{\min}}\right)^\frac{3p}{3p-2}
    +
    \frac{1}{C_{\min}}
    \right].
\end{equation*} 
\end{proof}
\smallskip

With the following result we extend the lower bound for $n$, obtained in Proposition \ref{prop:Bernis_2p-2-2p-1}.

\begin{proposition}\label{prop:Bernis_smaller-2p-2}
Let $\Omega\subset \R^d$ be a bounded, convex and smooth domain, and let $2\leq p <\frac{19}{3}$ be fixed. Assume that $u\in C^{\infty}(\bar{\Omega})$ is positive and satisfies $\nabla u \cdot \nu = 0$ on \(\partial \Omega\), and let \(\phi \in W^{1}_\infty(\Omega)\) be non-negative. If
\begin{equation*}
    2p - 2\frac{2 + \sqrt{(3p-2)^2 + d(3p-4)}}{-3p+6 + \sqrt{(3p-2)^2 + d(3p-4)}} < n < 2p - 2 \frac{6p-6 + \sqrt{(3p-2)^2 + d(3p-4)}}{9p-10+\sqrt{(3p-2)^2 + d(3p-4)}},
\end{equation*}  
then there exist positive constants \(C_i=C_i(n,p,d) >0,\ i=3,4,\)    
such that the following estimates hold:
\begin{equation}\label{eq:BernisLemma2.4_I}
    \int_{\Omega} \phi^{3p}u^{n-2p} |\nabla u|^{3p} \dd x
    \leq
    C_3\left[ \int_{\Omega} \phi^{3p} u^n|\nabla \Delta u|^p \dd x + \int_{\Omega} |\nabla \phi|^{3p} u^{n+p} \dd x\right]
\end{equation}
and 
\begin{equation}\label{eq:BernisLemma2.4_II}
    \begin{split}
        & 
        \int_{\Omega} \phi^{3p} u^{n-2p+2} |\nabla u|^{3p-6} |D^2u \nabla u|^2 \dd x
        +
        \int_{\Omega} \phi^{3p} u^{n-2p+2} |\nabla u|^{3p-4} |D^2u|^2 \dd x
        \\
        &
        + \int_{\partial\Omega}\phi^{3p}  u^{n-2p+2}|\nabla u|^{3p-4} \mathrm{II}(\nabla u,\nabla u) \dd \mathcal{H}^{d-1} 
        \leq\ 
        C_4\left[ \int_{\Omega} \phi^{3p} u^n|\nabla \Delta u|^p \dd x + \int_{\Omega} |\nabla \phi|^{3p} u^{n+p} \dd x\right].
    \end{split}
\end{equation}
\end{proposition}
\smallskip

\begin{remark}
\begin{itemize}
    \item[(i)] Observe that Proposition \ref{prop:Bernis_smaller-2p-2} is valid  for any dimension $d \geq 1$. Moreover, as  
    \begin{equation*}
        2p - 2 \frac{2 + \sqrt{(3p-2)^2 + d(3p-4)}}{-3p+6+\sqrt{(3p-2)^2 + d(3p-4)}} < 2p-2
    \end{equation*}
    for all \(p\geq 2\), Proposition \ref{prop:Bernis_smaller-2p-2} extends the statement of Proposition \ref{prop:Bernis_2p-2-2p-1} to values of $n$ that are slightly smaller than $2p-2$. However, in Proposition \ref{prop:Bernis_smaller-2p-2} the constants $C_3, C_4 > 0$ depend not only on $n$ and $p$, but also on the dimension $d$.
    \item[(ii)] Note that for $p=2$, the upper bound 
    \begin{equation*}
        2p - 2 \frac{6p-6 + \sqrt{(3p-2)^2 + d(3p-4)}}{9p-10+\sqrt{(3p-2)^2 + d(3p-4)}}
    \end{equation*}
    equals $2 + \frac{2}{4+\sqrt{\frac{8+d}{8}}}$, i.e. it coincides with the one in \cite{gruen2001}.
\end{itemize}
\end{remark}
\smallskip

\begin{proof}
We start by introducing the notation 
\begin{equation*}
    \begin{split}
         \mathcal{K}_1 & \coloneqq \int_{\Omega} \phi^{3p}  u^{n-2p}|\nabla u|^{3p} \dd x,\\
        \mathcal{K}_2 & \coloneqq \int_{\Omega} \phi^{3p-2}|\nabla \phi|^2 u^{n-2p+2} |\nabla u|^{3p-2} \dd x,\\
        \mathcal{K}_3 & \coloneqq \int_{\Omega} \phi^{3p} u^{n-2p+2} |\nabla u|^{3p-6} |D^2u \nabla u|^2 \dd x, \\
        \mathcal{K}_4 & \coloneqq \int_{\Omega} \phi^{3p} u^{n-2p+2} |\nabla u|^{3p-4} |D^2 u|^2 \dd x, \\
         \mathcal{B} & \coloneqq \int_{\partial \Omega} \phi^{3p} u^{n-2p+2} |\nabla u|^{3p-4}\mathrm{II}(\nabla u, \nabla u) \dd \mathcal{H}^{d-1}.
    \end{split}
\end{equation*}
Using this, we can rewrite \eqref{eq:BernisLemma21} from Lemma \ref{lem:BernisLemma21} as 
\begin{equation}\label{eq:KB1}
    \begin{split}
        (3p-4) \mathcal{K}_3 + \mathcal{K}_4 + \mathcal{B}
        = & -3p \int_{\Omega} \phi^{3p-1} u^{n-2p+2}|\nabla u|^{3p-4} \nabla\phi \cdot D^2u \nabla u \dd x 
        \\
        & - (n-2p+2) \int_{\Omega} \phi^{3p} u^{n-2p+1} |\nabla u|^{3p-4} \nabla u \cdot D^2u\nabla u \dd x 
        \\ &
        -\int_{\Omega} \phi^{3p} u^{n-2p+2} |\nabla u|^{3p-4} \nabla u \cdot \nabla \Delta u \dd x .
    \end{split}
\end{equation}    
We want to estimate the right-hand side of \eqref{eq:KB1} in terms of $\mathcal{K}_3$ and  $\mathcal{K}_4$. For that purpose, we apply the Cauchy--Schwarz and  Hölder inequalities, obtaining that 
\begin{equation}\label{eq:KB2}
    \begin{split}
        (3p-4) \mathcal{K}_3 + \mathcal{K}_4 + \mathcal{B}
        \leq &\ 3p \sqrt{\mathcal{K}_2}\sqrt{\mathcal{K}_3} 
        + |n-2p+2| \sqrt{\mathcal{K}_1}\sqrt{\mathcal{K}_3}   
        \\& 
        + \mathcal{K}_1^{\frac{p-1}{p}} \left(\int_{\Omega} \phi^{3p} u^n |\nabla \Delta u|^p\dd x\right)^{\frac{1}{p}}.
    \end{split}
\end{equation}   
{\bf Step 1. Estimates for $\mathcal{K}_1$ and $\mathcal{K}_2$ in terms of $\mathcal{K}_3$ and $\mathcal{K}_4$.} 
First, we observe that, by Hölder's inequality, $\mathcal{K}_2$ can be estimated in terms of $\mathcal{K}_1$ by
\begin{equation}\label{eq:K2}
    \mathcal{K}_2 \leq \mathcal{K}_1^{\frac{3p-2}{3p}} \left(\int_{\Omega}|\nabla \phi|^{3p} u^{n+p} \dd x\right)^{\frac{2}{3p}}.
\end{equation} 
For the integral  $\mathcal{K}_1$, we obtain from equation  \eqref{eq:proof-BernisProp23-1.}, by using the Cauchy--Schwarz inequality and the estimate
    \begin{equation*}
        \frac{1}{d} |\Delta u|^2 \leq \sum_{i=1}^d (\partial_{ii} u)^2 \leq \sum_{i,j=1}^d (\partial_{ij} u)^2 =|D^2u|^2 \quad \text{ in } \Omega,
    \end{equation*}
that  
    \begin{equation*}
    \begin{split}
       \mathcal{K}_1
        & = - \frac{3p}{n-2p+1} \int_{\Omega}\phi^{3p-1} u^{n-2p+1}|\nabla u|^{3p-2} \nabla \phi \cdot \nabla u \dd x 
        \\
        & \quad - \frac{3p-2}{n-2p+1} \int_{\Omega} \phi^{3p} u^{n-2p+1} |\nabla u|^{3p-4}  D^2 u\nabla u \cdot \nabla u \dd x \\
        & \quad - \frac{1}{n-2p+1}\int_{\Omega} \phi^{3p} u^{n-2p+1} |\nabla u|^{3p-2} \Delta u \dd x \\
        & \leq \frac{3p}{|n-2p+1|}  \sqrt{\mathcal{K}_1}\sqrt{\mathcal{K}_2}   + \frac{3p-2}{|n-2p+1|}  \sqrt{\mathcal{K}_1}\sqrt{\mathcal{K}_3}  
        + \frac{\sqrt{d}}{|n-2p+1|} \sqrt{\mathcal{K}_1}\sqrt{\mathcal{K}_4}, 
    \end{split}
    \end{equation*}    
and dividing both sides by  $\sqrt{\mathcal{K}_1}$  gives 
\begin{equation}\label{eq:KB3}
\sqrt{\mathcal{K}_1} \leq  \frac{1}{|n-2p+1|}  \left( 3p \sqrt{\mathcal{K}_2} + (3p-2) \sqrt{\mathcal{K}_3}  + \sqrt{d} \sqrt{\mathcal{K}_4} \right).
\end{equation}
Hence, by using Young's inequality, we see that 
\begin{equation}\label{eq:KB4}
    \begin{split}
        \mathcal{K}_1 & \leq \frac{1}{|n-2p+1|^2} \left(3p \sqrt{\mathcal{K}_2} + (3p-2) \sqrt{\mathcal{K}_3} + \sqrt{d}\sqrt{\mathcal{K}_4} \right)^2 \\
        & \leq \frac{18p^2}{|n-2p+1|^2} \mathcal{K}_2 +\frac{2}{|n-2p+1|^2}\left((3p-2) \sqrt{\mathcal{K}_3} + \sqrt{d}\sqrt{\mathcal{K}_4} \right)^2 .
    \end{split}
\end{equation}
Now, inserting the inequality \eqref{eq:K2} for $\mathcal{K}_2$ into \eqref{eq:KB4} and applying Young's inequality with weight $\tfrac{1}{2}$, we deduce the estimate
\begin{equation*}
    \begin{split}
       \mathcal{K}_1  
        & \leq 
        \frac{1}{2} \mathcal{K}_1 
        + 
        \frac{2}{3p}\left(\frac{3p}{6p-4}\right)^\frac{2-3p}{2} \left(\frac{18 p^2}{|n-2p+1|^2}\right)^{\frac{3p}{2}}
        \int_{\Omega} |\nabla\phi|^{3p} u^{n+p} \dd x \\
        & \quad +  \frac{2}{|n-2p+1|^2}\left((3p-2) \sqrt{ \mathcal{K}_3} + \sqrt{d}\sqrt{ \mathcal{K}_4} \right)^2.
    \end{split}
\end{equation*}
This directly implies the following estimate for $ \mathcal{K}_1 $:
\begin{equation}\label{eq:KB5}
    \mathcal{K}_1 \leq \frac{4}{|n-2p+1|^2}\left((3p-2) \sqrt{ \mathcal{K}_3} + \sqrt{d}\sqrt{ \mathcal{K}_4} \right)^2
    + C
    \int_{\Omega} |\nabla\phi|^{3p} u^{n+p} \dd x,
\end{equation}
where the constant $C$ is given by
\begin{equation*}
    C=C(n,p)
    \coloneqq
    \frac{4}{3p}\left(\frac{3p}{6p-4}\right)^\frac{2-3p}{2} \left(\frac{18 p^2}{|n-2p+1|^2}\right)^{\frac{3p}{2}}>0.
\end{equation*}
{\bf Step 2. Further estimate inequality \eqref{eq:KB2}.} 
Applying in \eqref{eq:KB2} the estimate \eqref{eq:KB3} for $\sqrt{\mathcal{K}_1}$ on the second term and Young's inequality with weight $ \tilde{\eps}_1$ on the third term, we find that
\begin{equation*}
    \begin{split}
        (3p-4) \mathcal{K}_3 + \mathcal{K}_4 + \mathcal{B}
        &\leq  3p \sqrt{\mathcal{K}_2}\sqrt{\mathcal{K}_3} 
        + \frac{|n-2p+2|}{|n-2p+1|} \left( 3p \sqrt{\mathcal{K}_2}\sqrt{\mathcal{K}_3} + (3p-2)  \mathcal{K}_3 + \sqrt{d} \sqrt{\mathcal{K}_3}\sqrt{\mathcal{K}_4}\right)
        \\  & \quad
        + \tilde{\eps}_1\mathcal{K}_1 + C_{ \tilde{\eps}_1}  \int_{\Omega} \phi^{3p} u^n |\nabla \Delta u|^p\dd x
        \\ &
        = 3p \left( 1+ \frac{|n-2p+2|}{|n-2p+1|}\right)  \sqrt{\mathcal{K}_2}\sqrt{\mathcal{K}_3}  + \frac{|n-2p+2|}{|n-2p+1|} \left( (3p-2)  \mathcal{K}_3 + \sqrt{d} \sqrt{\mathcal{K}_3}\sqrt{\mathcal{K}_4}\right)
         \\  & \quad
        + \tilde{\eps}_1\mathcal{K}_1 + C_{ \tilde{\eps}_1}  \int_{\Omega} \phi^{3p} u^n |\nabla \Delta u|^p\dd x
    \end{split}
\end{equation*}   
with 
\begin{equation*}
    \tilde{\eps}_1\coloneqq \eps_1 \frac{(n-2p+1)^2}{4} >0 \quad \text{and} \quad C_{ \tilde{\eps}_1} \coloneqq \frac{1}{p} \left( \frac{p}{p-1}  \tilde{\eps}_1\right)^{1-p} > 0
\end{equation*}
for some arbitrary constant $\eps_1>0$. By Young's inequality with weight $\frac{\eps_2}{2}$ for the second term on the right-hand side of the above inequality and \eqref{eq:KB5}, we further obtain
\begin{equation}\label{eq:KB6}
    \begin{split}
        (3p-4) \mathcal{K}_3 + \mathcal{K}_4 + \mathcal{B}
        &\leq 
         3p \left( 1+ \frac{|n-2p+2|}{|n-2p+1|}\right)  \sqrt{\mathcal{K}_2}\sqrt{\mathcal{K}_3}  
         \\  & \quad
         + \frac{|n-2p+2|}{|n-2p+1|} \left[ \left( (3p-2)  + \frac{\sqrt{d}}{2\eps_2}\right)\mathcal{K}_3 +\frac{\eps_2 \sqrt{d}}{2} \mathcal{K}_4\right]
         \\  & \quad
        + \eps_1 \left( (3p-2) \sqrt{\mathcal{K}_3} +\sqrt{d} \sqrt{\mathcal{K}_4}\right)^2 
        \\  & \quad
        + \tilde{\eps}_1  C \int_{\Omega} |\nabla\phi|^{3p} u^{n+p} \dd x
        \\  & \quad
        + C_{ \tilde{\eps}_1}  \int_{\Omega} \phi^{3p} u^n |\nabla \Delta u|^p\dd x
    \end{split}
\end{equation}   
for some arbitrary $\eps_2>0$.\\
Now we estimate the first term on the right-hand side of \eqref{eq:KB6}. Recalling the estimate \eqref{eq:K2} for $\mathcal{K}_2$ and applying twice Young's inequality with weights $\frac{\eps_3}{2}$ and $\eps_4$, we have
\begin{equation}\label{eq:KB7}
      \begin{split}
        &3p \left( 1+ \frac{|n-2p+2|}{|n-2p+1|}\right)  \sqrt{\mathcal{K}_2}\sqrt{\mathcal{K}_3}  \\
        & \leq 3p \left( 1+ \frac{|n-2p+2|}{|n-2p+1|}\right) \mathcal{K}_1^{\frac{3p-2}{6p}}\left( \int_{\Omega} |\nabla\phi|^{3p} u^{n+p} \dd x\right)^{\frac{1}{3p}} \sqrt{\mathcal{K}_3}\\
        & \leq \frac{\eps_3}{2} \mathcal{K}_3 + \frac{9p^2}{2\eps_3}
        \left( 1+ \frac{|n-2p+2|}{|n-2p+1|}\right)^2 \mathcal{K}_1^{\frac{3p-2}{3p}}\left( \int_{\Omega} |\nabla\phi|^{3p} u^{n+p} \dd x\right)^{\frac{2}{3p}}\\
        & \leq \frac{\eps_3}{2} \mathcal{K}_3 + \eps_4   \mathcal{K}_1 + \left(\frac{3p}{3p-2} \eps_4\right)^{\frac{2-3p}{2}} \frac{2}{3p} \left[ \frac{9p^2}{2\eps_3}
        \left( 1+ \frac{|n-2p+2|}{|n-2p+1|}\right)^2 \right]^{\frac{3p}{2}}\int_{\Omega} |\nabla\phi|^{3p} u^{n+p} \dd x
      \end{split}
\end{equation}
with some $\eps_3,\eps_4>0$.
Inserting \eqref{eq:KB7} and \eqref{eq:KB5} back into \eqref{eq:KB6} and rearranging, we have 
\begin{equation}\label{eq:KB8}
    \begin{split}
        &(3p-4) \mathcal{K}_3 + \mathcal{K}_4 + \mathcal{B}\\
        &\leq 
        \frac{\eps_3}{2} \mathcal{K}_3 + \left( \frac{4\eps_4}{|n-2p+1|^2} + \eps_1\right) \left((3p-2) \sqrt{ \mathcal{K}_3} + \sqrt{d}\sqrt{ \mathcal{K}_4} \right)^2\\
        & \quad 
         + \frac{|n-2p+2|}{|n-2p+1|} \left[ \left( (3p-2)  + \frac{\sqrt{d}}{2\eps_2}\right)\mathcal{K}_3 +\frac{\eps_2 \sqrt{d}}{2} \mathcal{K}_4\right]
         \\ 
        & \quad+ \left[ (\eps_4+\tilde{\eps}_1)  C + \left(\frac{3p}{3p-2} \eps_4\right)^{\frac{2-3p}{2}} \frac{2}{3p} \left[ \frac{9p^2}{2\eps_3}
        \left( 1+ \frac{|n-2p+2|}{|n-2p+1|}\right)^2 \right]^{\frac{3p}{2}}\right] \int_{\Omega} |\nabla\phi|^{3p} u^{n+p} \dd x\\  
        & \quad
        + C_{ \tilde{\eps}_1}  \int_{\Omega} \phi^{3p} u^n |\nabla \Delta u|^p\dd x.
    \end{split}
\end{equation}   
{\bf Step 3. Finish the proof.} 
We now want to absorb the terms on the right-hand side of \eqref{eq:KB8} into its left-hand side. Therefore, we choose \(\eps_2\) as the unique positive solution of
    \begin{equation*}
        (3p-2) + \frac{\sqrt{d}}{ 2\eps_2} = (3p-4) \frac{\eps_2\sqrt{d}}{2},
    \end{equation*}
since this absorbs the second line on the right-hand side of \eqref{eq:KB8} in an optimal way and also leads us to the desired bound for $n$. 
Solving this equation for \(\eps_2>0\) gives
\begin{equation*}
    \eps_2 
    = \frac{3p-2 + \sqrt{(3p-2)^2 + d(3p-4)}}{\sqrt{d}(3p-4)}.
\end{equation*}
If
\begin{equation}\label{eq:KB9}
    \frac{|n-2p+2|}{|n-2p+1|}  \eps_2 \sqrt{d} < 2,
\end{equation}
then we may absorb the second line on the right-hand side into the left-hand side. Inserting the equation for \(\eps_2\) into \eqref{eq:KB9} gives us a bound for \(n\) in terms of \(p\) and $d$. Indeed, the inequality
\begin{equation*}
    \frac{|n-2p+2|}{|n-2p+1|} \frac{3p-2 + \sqrt{(3p-2)^2 + d(3p-4)}}{3p-4} 
    < 2
\end{equation*}
is satisfied for
\begin{equation*}
    2p - 2\frac{2 + \sqrt{(3p-2)^2 + d(3p-4)}}{-3p+6 + \sqrt{(3p-2)^2 + d(3p-4)}} < n < 2p - 2 \frac{6p-6 + \sqrt{(3p-2)^2 + d(3p-4)}}{9p-10+\sqrt{(3p-2)^2 + d(3p-4)}}.
\end{equation*} 
Next, using Young's inequality in the second term on the right-hand side of \eqref{eq:KB8} and applying \eqref{eq:KB9}, we find that \eqref{eq:KB8} can be further estimated as follows: 
\begin{equation}\label{eq:KB10}
    \begin{split}
        &(3p-4) \mathcal{K}_3 + \mathcal{K}_4 + \mathcal{B}\\
        &\leq 
        2 \left[\frac{\eps_3}{2}  + \left( \frac{4\eps_4}{|n-2p+1|^2} + \eps_1\right) (3p-2) \left((3p-2) + \sqrt{d}\right) \right]\mathcal{K}_3\\
        & \quad + 2 \left( \frac{4\eps_4}{|n-2p+1|^2} + \eps_1\right)\sqrt{d}
        \left((3p-2) + \sqrt{d}\right) \mathcal{K}_4
         \\ 
        & \quad+ 2 \left[ (\eps_4+\tilde{\eps}_1)  C + \left(\frac{3p}{3p-2} \eps_4\right)^{\frac{2-3p}{2}} \frac{2}{3p} \left[ \frac{9p^2}{2\eps_3}
        \left( 1+ \frac{|n-2p+2|}{|n-2p+1|}\right)^2 \right]^{\frac{3p}{2}}\right] \int_{\Omega} |\nabla\phi|^{3p} u^{n+p} \dd x\\  
        & \quad
        + 2C_{ \tilde{\eps}_1}  \int_{\Omega} \phi^{3p} u^n |\nabla \Delta u|^p\dd x.
    \end{split}
\end{equation} 
Now we can choose $\eps_1,\eps_3$ and $\eps_4$ such that we can absorb the first two lines on the right-hand side of \eqref{eq:KB10} into its left-hand side.  Hence, we find a constant \(\tilde{C} = \tilde{C}(n,p,d)>0\) such that
    \begin{equation*}
       \mathcal{K}_3 + \mathcal{K}_4 + \mathcal{B} \leq \tilde{C}\left[ \int_{\Omega} |\nabla \phi |^{3p} u^{n+p} \dd x + \int_{\Omega}  \phi^{3p} u^n |\nabla \Delta u|^p \dd x\right].
    \end{equation*}
Combining this estimate with \eqref{eq:KB5} concludes the proof. 
\end{proof}

The final Bernis-type inequality, presented in the following proposition, provides an estimate for an integral involving the product of the function $u$ and its Laplacian. The proof does not require a case-by-case analysis based on the mobility exponent $n$ but can be directly established for the entire range of  $n$. The key ingredients of the proof are the inequalities \eqref{eq:BernisLemma23-1} and \eqref{eq:BernisLemma2.4_I}, which control the integral $\int_\Omega \varphi^{3p} u^{n-2p} |\nabla u|^{3p} \dd x$. Additionally, we note that the corresponding result for the one-dimensional case is presented in \cite{ansini2004}.

\begin{proposition}\label{prop:Bernis_new_III}
Let $\Omega\subset \R^d$ be a bounded, convex and smooth domain, and let $2 \leq p < \frac{19}{3}$ be fixed. Assume that $u\in C^{\infty}(\bar{\Omega})$ is positive and satisfies $\nabla u \cdot \nu = 0$ on \(\partial \Omega\), and let \(\phi \in W^{1}_\infty(\Omega)\) be non-negative. If
\begin{equation*}
    2p - 2\frac{2 + \sqrt{(3p-2)^2 + d(3p-4)}}{-3p+6 + \sqrt{(3p-2)^2 + d(3p-4)}} < n < 2p - 1,
\end{equation*}  
then there exists a positive constant \(C_5=C_5(n,p,d) >0\)  
such that the following estimate holds:
\begin{equation}\label{eq:Bernis_new_IV}
\begin{split}
    \int_{\Omega} \phi^{3p}u^{n-\frac{p}{2}} |\Delta u|^{\frac{3p}{2}} \dd x 
    \leq\ 
    C_5\left( \int_{\Omega} \phi^{3p} u^n|\nabla \Delta u|^p \dd x + \int_{\Omega} |\nabla \phi|^{3p} u^{n+p} \dd x\right).
\end{split}
\end{equation}
\end{proposition}

\begin{proof}
Using integration by parts, we can rewrite
\begin{equation} \label{eq:Bernis_Laplace_I}
\begin{split}
    \int_\Omega \varphi^{3p} u^{n-\frac{p}{2}} |\Delta u|^\frac{3p}{2} \dd x
    &=
    -3p\int_\Omega \varphi^{3p-1} u^{n-\frac{p}{2}} |\Delta u|^\frac{3p-4}{2} \Delta u \nabla \varphi\cdot \nabla u \dd x
    \\
    &\quad
    - \left(n-\frac{p}{2}\right) 
    \int_\Omega \varphi^{3p} u^{n-\frac{p}{2}-1} |\nabla u|^2 |\Delta u|^\frac{3p-4}{2} \Delta u \dd x
    \\
    &\quad
    - \left(\frac{3p-2}{2}\right) \int_\Omega 
    \varphi^{3p} u^{n-\frac{p}{2}} |\Delta u|^\frac{3p-4}{2} \nabla u\cdot \nabla\Delta u \dd x.
\end{split}
\end{equation}
Applying now H\"older's and Young's inequalities on each of the integrals on the right-hand side yields the estimates 
\begin{equation} \label{eq_Bernis_NEW_I}
\begin{split}
    &\int_\Omega \varphi^{3p-1} u^{n-\frac{p}{2}} |\Delta u|^\frac{3p-4}{2} \Delta u \nabla \varphi\cdot \nabla u \dd x
    \\
    \leq\ &
    \left(
    \int_\Omega
    \varphi^{3p} u^{n-\frac{p}{2}} |\Delta u|^\frac{3p}{2} \dd x
    \right)^\frac{3p-2}{3p}
    \left(
    \int_\Omega
    \varphi^{3p} u^{n-2p} |\nabla u|^{3p} \dd x
    \right)^\frac{1}{3p}
    \left(
    \int_\Omega
    |\nabla \varphi|^{3p} u^{n+p} \dd x
    \right)^\frac{1}{3p}
    \\
    \leq\ &
    \eps_1 \int_\Omega
    \varphi^{3p} u^{n-\frac{p}{2}} |\Delta u|^\frac{3p}{2} \dd x
    +
    C_{\eps_1}
    \int_\Omega
    \varphi^{3p} u^{n-2p} |\nabla u|^{3p} \dd x
    +
    \tilde{C}_{\eps_1}
    \int_\Omega
    |\nabla \varphi|^{3p} u^{n+p} \dd x,
\end{split}
\end{equation}
for the first integral,
\begin{equation}\label{eq_Bernis_NEW_II}
\begin{split}
    \int_\Omega \varphi^{3p} u^{n-\frac{p}{2}-1} |\nabla u|^2 |\Delta u|^\frac{3p-4}{2} \Delta u \dd x
    &\leq 
    \left(
    \int_\Omega
    \varphi^{3p} u^{n-\frac{p}{2}} |\Delta u|^\frac{3p}{2} \dd x
    \right)^\frac{3p-2}{3p}
    \left(
    \int_\Omega
    \varphi^{3p} u^{n-2p} |\nabla u|^{3p} \dd x
    \right)^\frac{2}{3p}
    \\
    &\leq
    \eps_2 \int_\Omega
    \varphi^{3p} u^{n-\frac{p}{2}} |\Delta u|^\frac{3p}{2} \dd x
    +
    C_{\eps_2}
    \int_\Omega
    \varphi^{3p} u^{n-2p} |\nabla u|^{3p} \dd x
\end{split}
\end{equation}
for the second integral, and
\begin{equation} \label{eq_Bernis_NEW_III}
\begin{split}
    & \int_\Omega 
    \varphi^{3p} u^{n-\frac{p}{2}} |\Delta u|^\frac{3p-4}{2} \nabla u\cdot \nabla\Delta u \dd x
    \\
    \leq\ 
    &
    \left(
    \int_\Omega
    \varphi^{3p} u^{n-\frac{p}{2}} |\Delta u|^\frac{3p}{2} \dd x
    \right)^\frac{3p-4}{3p}
    \left(
    \int_\Omega
    \varphi^{3p} u^{n-2p} |\nabla u|^{3p} \dd x
    \right)^\frac{1}{3p}
    \left(
    \int_\Omega \varphi^{3p} u^n |\nabla \Delta u|^p \dd x
    \right)^\frac{1}{p}
    \\
    \leq\ &
    \eps_3 \int_\Omega
    \varphi^{3p} u^{n-\frac{p}{2}} |\Delta u|^\frac{3p}{2} \dd x
    +
    C_{\eps_3} \int_\Omega
    \varphi^{3p} u^{n-2p} |\nabla u|^{3p} \dd x
    +
    \tilde{C}_{\eps_3}
    \int_\Omega \varphi^{3p} u^n |\nabla \Delta u|^p \dd x
\end{split}
\end{equation}
for the third one, where we chose $0 < \eps_1+\eps_2+\eps_3 < 1$. By the inequalities \eqref{eq:BernisLemma23-1}, respectively \eqref{eq:BernisLemma2.4_I} we can estimate the integral $\int_\Omega \varphi^{3p} u^{n-2p} |\nabla u|^{3p} \dd x$ in \eqref{eq_Bernis_NEW_I}--\eqref{eq_Bernis_NEW_III} in terms of $\int_\Omega \varphi^{3p} u^n |\nabla\Delta u|^p \dd x$ and $\int_\Omega |\nabla\varphi|^{3p} u^{n+p} \dd x$. This yields the desired inequality.
\end{proof}

To complete the proof of Theorem \ref{thm:Bernis}, it remains to perform the approximation argument to establish the result in the desired functional setting, i.e., for functions $u \in W^2_p(\Omega)$ with $2\leq p < \frac{19}{3}$. Note that all results in this section are valid for arbitrary dimensions $d$. The restriction $d < 3p$ arises only during the approximation argument, where Sobolev embeddings are employed.

\begin{proof}[\textbf{Proof of Theorem \ref{thm:Bernis}}] 
The proof follows the idea of the proof of \cite[Theorem 1.1]{gruen2001}, using  elliptic regularity theory in $L_p(\Omega)$ instead of $L_2(\Omega)$.
Let $ u \in W^2_p(\Omega)$, with $2 \leq p < \frac{19}{3}$, be a positive function satisfying \( \nabla u \cdot \nu = 0 \) on \( \partial \Omega \), 
\begin{equation} \label{eq:finite_integrals}
    \int_{\Omega} u^n |\nabla \Delta u|^p\, dx < \infty.
\end{equation}
We may assume that there is $\delta >0$ such that $u \geq  \delta > 0$ a.e. in $\Omega$. Using  \eqref{eq:finite_integrals}, this implies $\int_\Omega |\nabla \Delta u|^p\, dx < \infty$.

To prove the claim, we construct a sequence of functions \((u_k)_{k\in \N}\subset C^{\infty}(\bar{\Omega})\) with the following properties:
\begin{enumerate}
    \item $u_k \geq \tfrac{\delta}{2}$ in $\bar{\Omega}$ for $k \in \N$ sufficiently large;
    \item $\nabla u_k \cdot \nu = 0$ on $\partial \Omega$;
    \item $u_k \longrightarrow u$ strongly in $W^2_p(\Omega)$, as $k\to \infty$;
    \item $\nabla \Delta u_k \longrightarrow \nabla \Delta u$ strongly in $L_p(\Omega)$, as $k\to \infty$.
\end{enumerate}
Once this sequence $u_k$ is constructed, we know it satisfies the inequalities of Propositions \ref{prop:Bernis_2p-2-2p-1}, 
\ref{prop:Bernis_smaller-2p-2} and \ref{prop:Bernis_new_III} for the whole range
\begin{equation*}
    2p - 2\frac{2 + \sqrt{(3p-2)^2 + d(3p-4)}}{-3p+6 + \sqrt{(3p-2)^2 + d(3p-4)}} 
    < n 
    < 
    2p-1.
\end{equation*}
We can then take the limit on both sides of the corresponding inequalities. For the right-hand side, we apply Vitali's convergence theorem, while for the left-hand side, we use Fatou's lemma. This yields the corresponding estimates 
\eqref{eq:Bernis_new_I}, \eqref{eq:Bernis_new_II} and \eqref{eq:Bernis_new_III} for the limit function $u$. \\

We now construct the sequence $u_k$. Since $\nabla \Delta u \in L_p(\Omega)$, we have $\Delta u \in W^1_p(\Omega)$. Thus, there exists a sequence $(f_k)_{k\in \N} \subset C^{\infty}(\bar{\Omega})$ satisfying 
\begin{equation*}
    f_k \longrightarrow -\Delta u \quad \text{strongly in } W^1_p(\Omega) 
    \quad \text{and} \quad \int_{\Omega} f_k \dd x = 0,
\end{equation*}
see \cite{gilbargtrudinger}. Here, the zero-mean condition can be obtained since $\int_{\Omega} \Delta u \dd x = \int_{\partial\Omega} \nabla u\cdot \nu = 0$ by the Neumann boundary condition on $u$. For each $k \in \N$, let $u_k$ be the solution to the Neumann problem
\begin{equation*}
    \begin{cases}
        -\Delta u_k = f_k, \quad x\in \Omega\\
        \nabla u_k \cdot \nu = 0, \quad x\in \partial \Omega \\
        \int_{\Omega} u_k \dd x = \int_{\Omega} u \dd x.
    \end{cases}
\end{equation*}
By elliptic regularity theory \cite[Theorem 2.5.1.1, Remark 2.5.1.2]{grisvard2011}, we find that \(u_k\in C^{\infty}(\bar{\Omega})\) and  
\begin{equation*}
    u_k \longrightarrow u
    \quad \text{strongly in }
    W^3_p(\Omega).
\end{equation*}
In particular, $u_k$ satisfies (2), (3), and (4).  
Moreover, by the Sobolev embedding theorem, we have 
\begin{equation*}
    W^3_p(\Omega) \hookrightarrow W^2_q(\Omega) \quad 
    \text{for } 1 \leq q < \frac{dp}{d-p} 
    \quad \text{and} \quad
    W^2_q(\Omega) \hookrightarrow C(\bar{\Omega}) \quad
    \text{for } q > \frac{d}{2}.
\end{equation*}
Consequently, we find that
\begin{equation*}
    \begin{cases}
    u_k \longrightarrow u
    \quad \text{strongly in }
    W^2_q(\Omega)
    &
    \text{for } 1 \leq q < \frac{dp}{d-p}.       
    \\
    u_k \longrightarrow u
    \quad \text{strongly in }
    C(\bar{\Omega}) &
    \text{for } d < 3p.
    \end{cases}
\end{equation*}
The strong convergence $u_k \to u$ in $C(\bar{\Omega})$ and $u \geq \delta >0$ imply (1) for $k \in \N$ large enough. This concludes the construction of the sequence \( u_k \) satisfying properties (1)--(4), thereby completing the proof of the theorem.
\end{proof}



\section{Proof of Theorem \ref{thm:consequence}}
\label{sec:secconsequence}
We now prove Theorem \ref{thm:consequence} which is a consequence of the inequalities \eqref{eq:Bernis_new_I} and \eqref{eq:Bernis_new_III} of Theorem \ref{thm:Bernis}.

\begin{proof}[\textbf{Proof of Theorem \ref{thm:consequence}}]
The first parts of \eqref{eq:consequence}, respectively \eqref{eq:Bernis-inequality-more}, follow from \eqref{eq:Bernis_new_I}, while parts two and three follow from \eqref{eq:Bernis_new_I} and \eqref{eq:Bernis_new_III}. We mention again that we actually prove weighted versions of the first-order and second-order estimate in \eqref{eq:Bernis-inequality-more}.

\noindent\textbf{Parts one of \eqref{eq:consequence}, \eqref{eq:Bernis-inequality-more}: } We first show that $u^{\frac{n+p}{3p}}$  lies in $W^1_{3p}(\Omega)$ and that the inequality 
\begin{equation}\label{eq:consequence1}
        \int_{\Omega} \phi^{3p} \Big|\nabla \big(u^{\frac{n+p}{3p}}\big)\Big|^{3p} \dd x 
        \leq C \left(\int_{\Omega} \phi^{3p} u^n|\nabla \Delta u|^p \dd x + \int_{\Omega} |\nabla \phi|^{3p} u^{n+p} \dd x\right)
\end{equation}
holds for some constant \(C=C(n,p,d)>0\) and for all non-negative test functions $\phi \in W^1_\infty(\Omega)$ satisfying $\int_\Omega |\nabla \varphi|^{3p} u^{n+p} \dd x < \infty$. Since \(u \in W^2_p(\Omega)\) and 
$W^2_p(\Omega) \hookrightarrow L_{n+p}(\Omega)$,
we have \(u^{\frac{n+p}{3p}} \in L_{3p}(\Omega)\). In addition, we compute that 
\begin{equation}
\label{eq:consequence2}
    \nabla \big(u^{\tfrac{n+p}{3p}}\big) = \tfrac{n+p}{3p} \ u^{\tfrac{n-2p}{3p}} \nabla u,
\end{equation}
and together with  \eqref{eq:Bernis_new_I} from Theorem \ref{thm:Bernis}, we obtain that $u^{\frac{n+p}{3p}} \in W^1_{3p}(\Omega)$ with the inequality
\begin{equation*}
\begin{split}
    \int_{\Omega} \varphi^{3p} \Big|\nabla \big(u^{\frac{n+p}{3p}}\big)\Big|^{3p} \dd x 
    &= 
    \left( \tfrac{n+p}{3p} \right)^{3p} \int_{\Omega} \varphi^{3p} u^{n-2p} |\nabla u|^{3p} \dd x
    \\
    &\leq 
    \left( \tfrac{n+p}{3p} \right)^{3p}  c_1 
    \left(
    \int_{\Omega} \varphi^{3p} u^n|\nabla \Delta u|^p \dd x
    +
    \int_\Omega |\nabla\varphi|^{3p} u^{n+p} \dd x\right)
    < \infty.
\end{split}
\end{equation*}
Setting $\varphi \equiv 1$ yields the first estimate.

\medskip

\noindent\textbf{Parts two of \eqref{eq:consequence}, \eqref{eq:Bernis-inequality-more}: }
Next, we show that  $u^{\frac{2(n+p)}{3p}}$ belongs to $W^2_{\frac{3p}{2}}(\Omega)$ and that there exists a positive constant \(\tilde{C}=\tilde{C}(n,p,d)\) such that 
\begin{equation}\label{eq:consequence3}
        \int_{\Omega} \phi^{3p} \Big|\Delta \big(u^{\frac{2(n+p)}{3p}}\big)\Big|^{\frac{3p}{2}} \dd x 
        \leq \tilde{C} \left(\int_{\Omega} \phi^{3p} u^n|\nabla \Delta u|^p \dd x + \int_{\Omega} |\nabla \phi|^{3p} u^{n+p} \dd x\right)
\end{equation}
holds for all non-negative test functions $\phi \in W^1_\infty(\Omega)$ satisfying $\int_\Omega |\nabla \varphi|^{3p} u^{n+p} \dd x < \infty$.
Since \(u \in W^2_p(\Omega) \hookrightarrow L_{n+p}(\Omega) \), we obtain as before that \(u^{\frac{2(n+p)}{3p}} \in  L_{\frac{3p}{2}}(\Omega)\). For \(\nabla\big( u^{\frac{2(n+p)}{3p}}\big)\), we compute 
\begin{equation}\label{eq:consequence4}
   \nabla\big( u^{\frac{2(n+p)}{3p}}\big) = \tfrac{2(n+p)}{3p} u^{\frac{2n-p}{3p}} \nabla u
\end{equation}
and use H\"older's inequality together with \eqref{eq:Bernis_new_I} (with $\varphi \equiv 1$) to find that 
\begin{equation*}
\begin{split}
    \int_{\Omega} \Big|\nabla \big(u^{\frac{2(n+p)}{3p}}\big)\Big|^{\frac{3p}{2}} \dd x  
    &= \left( \tfrac{2(n+p)}{3p} \right)^{\frac{3p}{2}} \int_{\Omega} u^{\frac{n+p}{2}} u^{\frac{n}{2}-p}|\nabla u|^{\frac{3p}{2}} \dd x 
    \\
    & \leq 
    \left( \tfrac{2(n+p)}{3p} \right)^{\frac{3p}{2}} \, c_1^{\frac{1}{2}} \left(\int_{\Omega} u^{n+p}\dd x\right)^{\frac{1}{2}} \left(\int_{\Omega} u^n|\nabla \Delta u|^p \dd x\right)^{\frac{1}{2}}<\infty.
\end{split}
\end{equation*}
Hence, \(\nabla \big(u^{\frac{2(n+p)}{3p}}\big) \in L_{\frac{3p}{2}}(\Omega;\R^d)\). Now, we compute 
\[
\Delta \big(u^{\frac{2(n+p)}{3p}}\big) = \tfrac{2(n+p)}{3p}\Big( \tfrac{2n-p}{3p} u^{\tfrac{2(n-2p)}{3p}} |\nabla u|^2 + u^{\tfrac{2n-p}{3p}} \Delta u\Big),
\]
and by Jensen's inequality we further have that
\begin{equation}\label{eq:consequence5}
  \big|\Delta \big(u^{\frac{2(n+p)}{3p}}\big)\big|^{\frac{3p}{2}}  \leq c \left( u^{n-2p} |\nabla u|^{3p} + u^{n-\frac{p}{2}}|\Delta u|^{\frac{3p}{2}} \right)
\end{equation}
for some constant \( c=c(n,p)>0\). 
Using \eqref{eq:Bernis_new_I} and  \eqref{eq:Bernis_new_III}, we infer from \eqref{eq:consequence5} that \(\Delta \big(u^{\frac{2(n+p)}{3p}}\big)\) belongs to \( L_{\frac{3p}{2}}(\Omega)\) and that \eqref{eq:consequence3} holds. We deduce the second estimate in \eqref{eq:Bernis-inequality-more} by setting $\phi \equiv 1$ in \eqref{eq:consequence3}. 
It remains to check that \(D^2\big(u^{\frac{2(n+p)}{3p}}\big) \in L_{\frac{3p}{2}}(\Omega;\R^{d\times d}) \). It follows from \eqref{eq:consequence4} and the Neumann boundary condition for \(u\) that \(\nabla\big(u^{\frac{2(n+p)}{3p}}\big)\cdot \nu =0\) on \(\partial \Omega\). Elliptic $L_{p}$-theory (see for instance \cite[Thm. 2.4.2.7 \& 2.5.5.1]{grisvard2011}) then proves the claim.

\medskip

\noindent\textbf{Parts three of \eqref{eq:consequence}, \eqref{eq:Bernis-inequality-more}:}
Finally, we show that \( u^{\frac{n+p}{p}} \) belongs to \( W^3_{p}(\Omega) \) and that there exists a constant \(\hat{C}= \hat{C}(n,p,d)> 0\) such that 
\begin{equation}\label{eq:consequencepart3.1}
        \int_{\Omega} \bigl|\nabla \Delta \big(u^{\frac{n+p}{p}}\big)\bigr|^p \, \mathrm{d}x 
        \leq \hat{C} \int_{\Omega} u^n |\nabla \Delta u|^p \, \dd x.
\end{equation} 
Because \( u \in W^2_p(\Omega) \hookrightarrow L_{n+p}(\Omega) \), we have \( u^{\frac{n+p}{p}} \in L_p(\Omega) \). For the gradient of \( u^{\frac{n+p}{p}} \), we calculate
\begin{equation}\label{eq:consequencepart3.2}
   \nabla\big( u^{\frac{n+p}{p}} \big) = \tfrac{n+p}{p} u^{\frac{n}{p}} \nabla u,
\end{equation}
and using H\"older's inequality along with \eqref{eq:Bernis_new_I} for $\varphi \equiv 1$ we deduce that 
\begin{equation*}
\begin{split}
    \int_{\Omega} \big|\nabla \big(u^{\frac{n+p}{p}}\big)\big|^{p} \, \dd x  
    &= 
    \left(\tfrac{n+p}{p}\right)^p \int_{\Omega} u^{n} |\nabla u|^p \, \dd x \\
    & \leq
    \left(\tfrac{n+p}{p}\right)^p 
    \left(
    \int_\Omega u^{n+p} \dd x
    \right)^\frac{2}{3}
    \left(
    \int_{\Omega} u^{n-2p} |\nabla u|^{3p} \, \dd x
    \right)^\frac{1}{3}
    \\
    & \leq 
    \left(\tfrac{n+p}{p}\right)^p \, c_1^{\frac{1}{3}} \left(\int_{\Omega} u^{n+p} \, \mathrm{d}x \right)^{\frac{2}{3}} \left(\int_{\Omega} u^{n} |\nabla \Delta u|^{p} \, \mathrm{d}x \right)^{\frac{1}{3}} < \infty.
\end{split}
\end{equation*}
Thus, \(\nabla\big(u^{\frac{n+p}{p}}\big) \in L_{p}(\Omega; \mathbb{R}^d)\). For the Laplacian of 
\( u^{\frac{n+p}{p}} \), we compute
\[
   \Delta\big( u^{\frac{n+p}{p}} \big) 
   = 
   \tfrac{n+p}{p} u^{\frac{n}{p}} \Delta u 
   + 
   \tfrac{n(n+p)}{p^2} u^{\frac{n-p}{p}} |\nabla u|^2.
\]
By applying Jensen's and H\"older's inequalities, along with  \eqref{eq:Bernis_new_I} and \eqref{eq:Bernis_new_III}, again for $\varphi \equiv 1$, we obtain
\begin{equation*}
\begin{split}
    \int_{\Omega} \big|\Delta\big( u^{\frac{n+p}{p}} \big)\big|^{p} \, \dd x  
    &\leq c \left( \int_{\Omega} u^{n} |\Delta u|^p \, \mathrm{d}x + \int_{\Omega} u^{n-p} |\nabla u|^{2p} \, \mathrm{d}x \right) \\
    & \leq c \left(\int_{\Omega} u^{n+p} \, \dd x \right)^{\frac{1}{3}}\left[ \left(\int_{\Omega} u^{n-\frac{p}{2}} |\Delta u|^{\frac{3p}{2}} \, \mathrm{d}x \right)^{\frac{2}{3}}  + \left(\int_{\Omega} u^{n-2p} |\nabla u|^{3p} \, \mathrm{d}x \right)^{\frac{2}{3}}\right]\\
    & \leq c\left(c_1^{\frac{2}{3}}+c_3^{\frac{2}{3}}\right) \left(\int_{\Omega} u^{n+p} \, \mathrm{d}x \right)^{\frac{1}{3}}\left(\int_{\Omega} u^{n} |\nabla \Delta u|^{p} \, \mathrm{d}x \right)^{\frac{2}{3}} < \infty
\end{split}
\end{equation*}
for some constant $c=c(n,p)>0$. Hence, 
$\Delta \big( u^{\frac{n+p}{p}} \big) \in L_p(\Omega)$.
To verify that \(D^2\big(u^{\frac{n+p}{p}}\big)\) belongs to  \(L_{p}(\Omega;\R^{d\times d}) \), observe that from \eqref{eq:consequencepart3.2} and the Neumann boundary condition for \(u\), we have \(\nabla\big(u^{\frac{n+p}{p}}\big)\cdot \nu =0\) on \(\partial \Omega\). The claim then follows by applying elliptic $L_{p}$-theory (see, e.g., \cite[Thm. 2.4.2.7]{grisvard2011}).
Next, it follows from 
\begin{equation*}
    \nabla \Delta \big(u^{\frac{(n+p)}{p}}\big) 
    = 
    \tfrac{(n+p)n(n-p)}{p^3} \ u^{\frac{n-2p}{p}}|\nabla u|^2 \nabla u 
    + 
    \tfrac{2(n+p)n}{p^2} \ u^{\frac{n-p}{p}} D^2 u \nabla u + \tfrac{(n+p)n}{p^2} \ u^{\frac{n-p}{p}} \Delta u \nabla u + \tfrac{n+p}{p}\  u^{\frac{n}{p}} \nabla\Delta u 
\end{equation*}
and Jensen's inequality that there is a constant $\tilde{c}=\tilde{c}(n,p)>0$ such that
\begin{equation}\label{eq:consequencepart3.3}
|\nabla \Delta (u^{\frac{n+p}{p}})|^p \leq \tilde{c}\bigl(u^{n-2p} |\nabla u|^{3p}+ u^{n-p}|D^2u \nabla u|^p + u^{n-p}|\Delta u\nabla u|^p + u^n|\nabla \Delta u|^p\bigr).
\end{equation}
The main task is to estimate the second-order term $u^{n-p} |D^2 u\nabla u|^p$. To this end, we compute for $s\in \R\setminus \{0\}$ the following two identities: 
\begin{equation*}
D^2 u \nabla u = \tfrac{1}{s} u^{1-s} D^2 u^s \nabla u - (s-1) u^{-1} |\nabla u|^2 \nabla u \quad \text{and} \quad \Delta u \nabla u = \tfrac{1}{s} u^{1-s}\Delta u^s \nabla u - (s-1) u^{-1} |\nabla u|^2 \nabla u.
\end{equation*}
Multiplying these identities by $u^{\frac{n-p}{p}}$, we rewrite them as
\begin{equation}\label{eq:consequencepart3.4}
u^{\frac{n-p}{p}} D^2 u \nabla u = \tfrac{1}{s} u^{\frac{n}{p}-s} D^2 u^s \nabla u - (s-1)  u^{\frac{n-2p}{p}} |\nabla u|^2 \nabla u 
\end{equation}
and 
\begin{equation}\label{eq:consequencepart3.5}
u^{\frac{n-p}{p}} \Delta u \nabla u = \tfrac{1}{s} u^{\frac{n}{p}-s} \Delta u^s \nabla u - (s-1)  u^{\frac{n-2p}{p}} |\nabla u|^2 \nabla u.
\end{equation}
For the term $u^{\frac{n-2}{p}} |\nabla u|^2 \nabla u$, appearing in \eqref{eq:consequencepart3.4} and \eqref{eq:consequencepart3.5}, a straightforward calculation gives the identity
\begin{equation*}
    u^{\frac{n-2p}{p}} |\nabla u|^2 \nabla u =  \left(\tfrac{3p}{n+p} \right)^3 |\nabla \big( u^{\frac{n+p}{3p}} \big)|^2\  \nabla \big( u^{\frac{n+p}{3p}} \big).
\end{equation*}
Substituting this identity back into \eqref{eq:consequencepart3.4} and \eqref{eq:consequencepart3.5}, and then applying Jensen's inequality, results in 
\begin{equation*}
u^{n-p} |D^2 u \nabla u|^p \leq \hat{c} \left( u^{n-sp} |D^2 u^s \nabla u|^p + |\nabla \big( u^{\frac{n+p}{3p}} \big)|^{3p}  \right)
\end{equation*}
and 
\begin{equation*}
u^{n-p} |\Delta u \nabla u|^p \leq \hat{c} \left( u^{n-sp} |\Delta u^s \nabla u|^p + |\nabla \big( u^{\frac{n+p}{3p}} \big)|^{3p}  \right)
\end{equation*}
with a constant $\hat{c}>0$, depending only on $n$, $p$ and $s$. Setting \( s = \frac{2(n+p)}{3p}\), it follows from \eqref{eq:consequencepart3.3}, by applying the above inequalities and H\"older's inequality, along with \eqref{eq:Bernis_new_I} and the estimates from the first and second part of \eqref{eq:Bernis-inequality-more} with $\varphi \equiv 1$, that 
\begin{equation}\label{eq:consequencepart3.6}
\begin{split}
    &\int_{\Omega} \big|\nabla \Delta\big( u^{\frac{n+p}{p}} \big)\big|^{p} \, \mathrm{d}x  \\
    &\leq \tilde{c} \int_{\Omega} \bigl(u^{n-2p} |\nabla u|^{3p}+ u^{n-p}|D^2u \nabla u|^p + u^{n-p}|\Delta u\nabla u|^p + u^n|\nabla \Delta u|^p\bigr)\, \mathrm{d}x\\
    &\leq c_\ast \left(\int_{\Omega} u^{n} |\nabla \Delta u|^{p} \, \mathrm{d}x +\int_{\Omega} |\nabla \big( u^{\frac{n+p}{3p}}\big)|^{3p} \, \mathrm{d}x + \int_{\Omega} u^{\frac{n-2p}{3}}|\Delta \big( u^{\frac{2(n+p)}{3p}} \big)\nabla u|^p\, \mathrm{d}x \right.\\
    &\quad + \left. \int_{\Omega} u^{\frac{n-2p}{3}}|D^2\big( u^{\frac{2(n+p)}{3p}} \big)\nabla u|^p\, \mathrm{d}x \right)\\
    &\leq c_\ast  \int_{\Omega} u^{n} |\nabla \Delta u|^{p} \, \mathrm{d}x
    +c_\ast  \left(\int_{\Omega} u^{n-2p}| \nabla u|^{3p}\, \mathrm{d}x\right)^{\frac{1}{3}}
    \left(\int_{\Omega} |\Delta \big( u^{\frac{2(n+p)}{3p}} \big)|^{\frac{3p}{2}}\, \mathrm{d}x\right)^{\frac{2}{3}}\\
    &\quad + c_\ast  \left(\int_{\Omega} u^{n-2p}| \nabla u|^{3p}\, \mathrm{d}x\right)^{\frac{1}{3}}\left(\int_{\Omega} |D^2\big( u^{\frac{2(n+p)}{3p}} \big)|^{\frac{3p}{2}}\, \mathrm{d}x\right)^{\frac{2}{3}}\\
    &\leq c_\ast  \int_{\Omega} u^{n} |\nabla \Delta u|^{p} \, \mathrm{d}x + c_\ast  \left(\int_{\Omega}u^{n} |\nabla \Delta u|^{p}\, \mathrm{d}x\right)^{\frac{1}{3}}
    \left(\int_{\Omega} |D^2\big( u^{\frac{2(n+p)}{3p}} \big)|^{\frac{3p}{2}}\, \mathrm{d}x\right)^{\frac{2}{3}},
\end{split}
\end{equation}
where $c_\ast  > 0$ is a generic constant that may change from line to line, but depends only on $n$, $p$, and $d$.
Due to the fact that  \(D^2\big(u^{\frac{2(n+p)}{3p}}\big) \in L_{\frac{3p}{2}}(\Omega;\R^{d\times d}) \), we deduce  that 
$\nabla \Delta\big( u^{\frac{n+p}{p}} \big) \in L_p(\Omega;\R^d)$. Therefore, applying again elliptic \(L_p\)-theory (see \cite[Thm. 2.5.5.1]{grisvard2011}), we conclude that \(u^{\frac{n+p}{p}}  \in W^3_p(\Omega)\). Furthermore, using \cite[Thm. 2.3.3.6]{grisvard2011}, we obtain the estimate
\begin{equation}\label{eq:consequencepart3.7}
 \begin{split}
     \| D^2\big( u^{\frac{2(n+p)}{3p}} \big)\|_{L_{\frac{3p}{2}}(\Omega)} 
     &\leq 
    c^\ast 
     \Bigl( \| \Delta\big( u^{\frac{2(n+p)}{3p}} \big)\|_{L_{\frac{3p}{2}}(\Omega)} + \|  u^{\frac{2(n+p)}{3p}} \|_{L_{\frac{3p}{2}}(\Omega)}  
     \Bigr) \\
     &\leq c^\ast \Bigl( \| \Delta\big( u^{\frac{2(n+p)}{3p}} \big)\|_{L_{\frac{3p}{2}}(\Omega)} + \|  \nabla \big(u^{\frac{2(n+p)}{3p}}\big) \|_{L_{\frac{3p}{2}}(\Omega)}\Bigr)
 \end{split}
 \end{equation}
for some generic constant $c^\ast>0$ depending on $n$, $p$, and $\Omega$. To establish inequality \eqref{eq:consequencepart3.1} we insert the elliptic estimate \eqref{eq:consequencepart3.7} into \eqref{eq:consequencepart3.6} and apply Young's inequality as well as the first two statements of the theorem. This leads to
\begin{equation*}
\begin{split}
    \int_{\Omega} \big|\nabla \Delta\big( u^{\frac{n+p}{p}} \big)\big|^{p} \, \dd x 
    &\leq
    c_\ast \left( \int_\Omega u^n |\nabla \Delta u|^p \dd x
    +
    \int_\Omega |\Delta \big(u^\frac{2(n+p)}{3p}\big)|^\frac{3p}{2} \dd x + \int_\Omega |\nabla \big(u^\frac{2(n+p)}{3p}\big)|^{\frac{3p}{2}} \dd x
    \right)
    \\
    &\leq
    c_\ast \int_\Omega u^n |\nabla \Delta u|^p \dd x.
\end{split}
\end{equation*}
Thus, the proof is complete.
\end{proof}

We can improve the estimates in Theorem \ref{thm:consequence} to localized versions as follows.

\begin{corollary} \label{cor:improve_weighted}
    Under the assumptions of Theorem \ref{thm:consequence} we have the additional weighted estimates
    \begin{equation*} 
        \int_{\Omega} \phi^{3p} \Big|\nabla \big(u^{\frac{n+p}{3p}}\big)\Big|^{3p} \dd x + 
        \int_{\Omega} \phi^{3p} \Big|\Delta\big(u^{\frac{2(n+p)}{3p}}\big)\Big|^{\frac{3p}{2}} \dd x 
        \leq 
        c_5 \left(\int_{\Omega} \phi^{3p} u^n|\nabla \Delta u|^p \dd x + \int_{\Omega} |\nabla \phi|^{3p} u^{n+p} \dd x\right)
    \end{equation*}
    for all non-negative test functions $\phi \in W^1_\infty(\Omega)$ satisfying $\int_\Omega |\nabla \phi|^{3p} u^{n+p}\, dx < \infty$, 
    and where $c_5$ is a positive constant depending only on $n$, $p$, and $d$.
\end{corollary}

The proof of these two estimates is already contained in the proof of Theorem \ref{thm:consequence}.

\begin{remark}\label{rem:consequences}
In Corollary \ref{cor:improve_weighted}, we address the estimates for the first- and second-order terms. However, for the third-order term, we can only establish the following result. Assuming, in addition to the assumptions of Theorem \ref{thm:consequence}, that an elliptic estimate of the form
\begin{equation}\label{eq:additional_assumption}
    \int_\Omega \psi |D^2 v|^s \dd x
        \leq
        c_s \left(\int_\Omega \psi |\Delta v|^s \dd x
        +
        \int_\Omega |D^2 v|^{s-1} |\nabla \psi|\, |\nabla v| \dd x + \text{l.o.t.}\right),
\end{equation}
for $1 < s < \infty$ and suitable test functions $\psi \geq 0$, holds true for solutions $v$ to the Neumann problem
\begin{equation*}
    \begin{cases}
        -\Delta v + v = f \quad \text{in } \Omega
        \\
       \nabla v\cdot \nu = 0 \quad \text{on } \partial\Omega,
    \end{cases}
\end{equation*}
then we also get the weighted estimate
    \begin{equation} \label{eq:third-order-weighted}
        \int_{\Omega} \phi^{3p}  \Big|\nabla \Delta\big(u^{\frac{n+p}{p}}\big)\Big|^{p} \dd x 
        \leq 
        c_6 \left(\int_{\Omega} \phi^{3p} u^n|\nabla \Delta u|^p \dd x + \int_{\Omega} |\nabla \phi|^{3p} u^{n+p} \dd x\right)
    \end{equation}
for suitable non-negative test functions $\phi$ satisfying $\int_\Omega |\nabla \phi|^{3p} u^{n+p}\, dx < \infty$ and for some constant $c_6=c_6(n,p,d)>0$.
The authors are not aware whether an estimate of the form \eqref{eq:additional_assumption} holds true for the Neumann problem. 
For the Dirichlet case, weighted elliptic estimates are available (cf. \cite{CD2018,DST2008,DST2010}), provided the weight belongs to some Muckenhoupt class. In this case, the second integral on the right-hand side of \eqref{eq:additional_assumption} becomes superfluous. 
However, since the test functions $\phi$ in this paper might vanish on sets of positive measure, they do not belong to the Muckenhoupt class. As a result, \eqref{eq:additional_assumption}, if it holds at all, would require the additional integral involving lower-order terms. It is also possible that the test function might need $W^2_\infty$-regularity for \eqref{eq:additional_assumption} to be valid.
\end{remark}
\smallskip

\begin{proof}[Sketch of proof of \eqref{eq:third-order-weighted} under the assumption \eqref{eq:additional_assumption}. ]
If an estimate of the form \eqref{eq:additional_assumption} is satisfied, where we ignore the lower-order terms in the calculation, we can choose 
\begin{equation*}
    \psi = \varphi^{3p}, \quad 
    v = u^\frac{2(n+p)}{3p} 
    \quad \text{and} \quad 
    s = \frac{3p}{2}
\end{equation*}
and apply Hölder's and Young's inequalities to obtain
\begin{equation*}
    \begin{split}
        &\int_\Omega \varphi^{3p} |D^2(u^\frac{2(n+p)}{3p})|^\frac{3p}{2} \dd x
        \\
        &\leq
        c_p\int_\Omega \varphi^{3p} |\Delta(u^\frac{2(n+p)}{3p})|^\frac{3p}{2} \dd x
        +
        2c_p(n+p) \int_\Omega 
        \varphi^{3p-1} u^\frac{2n-p}{3p} |D^2(u^\frac{2(n+p)}{3p})|^\frac{3p-2}{2} |\nabla \varphi| |\nabla u| \dd x 
        \\
        &\leq
        c_p\int_\Omega \varphi^{3p} |\Delta(u^\frac{2(n+p)}{3p})|^\frac{3p}{2} \dd x\\
        &\quad + 
        2c_p(n+p)\left(\int_\Omega \varphi^{3p} |D^2(u^\frac{2(n+p)}{3p})|^\frac{3p}{2} \dd x \right)^\frac{3p-2}{3p}
        \left(\int_\Omega |\nabla \varphi|^{3p} u^{n+p} \dd x\right)^\frac{1}{3p}
        \left(\int_\Omega \varphi^{3p} u^{n-2p} |\nabla u|^{3p} \dd x\right)^\frac{1}{3p}
        \\
        &\leq
        c_p\int_\Omega \varphi^{3p} |\Delta(u^\frac{2(n+p)}{3p})|^\frac{3p}{2} \dd x
        + 
        \eps \int_\Omega \varphi^{3p} |D^2(u^\frac{2(n+p)}{3p})|^\frac{3p}{2} \dd x\\
        & \quad 
        +
        C_\eps \left(\int_\Omega |\nabla \varphi|^{3p} u^{n+p} \dd x
        +
        \int_\Omega \varphi^{3p} u^{n-2p} |\nabla u|^{3p} \dd x
        \right)
    \end{split}
\end{equation*}
for some $0 < \eps < 1$,  where $C_\eps>0$ depends only on $n$, $p$, and $\eps$. 
Absorbing the second term on the right-hand side in the left-hand side, estimating the first term by \eqref{eq:Bernis-inequality-more} and the fourth term by \eqref{eq:Bernis_new_I}, we obtain the weighted estimate 
\begin{equation}\label{eq:second-order-weighted}
 \int_\Omega \varphi^{3p} |D^2(u^\frac{2(n+p)}{3p})|^\frac{3p}{2} \dd x \leq  c_7 \left( \int_\Omega \varphi^{3p} u^{n} |\nabla \Delta u|^{p} \dd x + \int_\Omega |\nabla \varphi|^{3p} u^{n+p} \dd x
        \right)  
\end{equation}
with a constant $c_7 > 0$ depending on $n$, $p$, and $d$. Following the steps for \eqref{eq:consequencepart3.6} and using \eqref{eq:second-order-weighted} we finally obtain \eqref{eq:third-order-weighted}.
\end{proof}

\medskip

\section*{Acknowledgements}
\noindent
The authors would like to thank Jonas Jansen for fruitful discussions. 
\medskip

\section*{Funding}
\noindent 
KN was supported by ENW Vidi grant VI.Vidi.223.019 entitled 'Codimension two free boundary problems' of the Dutch Research Council NWO and by the APART-MINT Fellowship of the Austrian Academy of Sciences (\"OAW-APART-MINT Grant No. 12114)

\printbibliography

\end{document}